\documentclass[11pt,a4paper,reqno]{article}

\usepackage[utf8]{inputenc}
\usepackage[affil-it]{authblk}
\usepackage{amsfonts, amsmath, amssymb, amsthm, bbm, color, enumerate, graphicx, mathtools, tikz, hyperref, relsize}
\usepackage[numbers,sort&compress]{natbib}
\usepackage[sc]{mathpazo}
\usepackage[T1]{fontenc}
\usepackage{eulervm}
\usepackage[sc]{mathpazo}
\usepackage[a4paper,includeheadfoot, total={6.2in,9.5in}]{geometry}

\usepackage{subcaption} 

\usepackage{tikz-3dplot}
\usetikzlibrary{decorations.text, decorations.markings, mindmap,trees,shadows, decorations.pathreplacing}
\tikzstyle{doublearr}=[latex-latex,red, line width=0.5pt]
\tikzstyle{doublearr2}=[latex-latex,green!80!black, line width=0.5pt]
\usepackage{xcolor}

\newcommand{\bld}[1]{\boldsymbol{#1}}

\newcommand{\qq}{\boldsymbol{q}}
\newcommand{\dd}{\boldsymbol{\delta}}

\newcommand{\bx}{\bar{x}}
\newcommand{\ba}{\bar{a}}
\newcommand{\bd}{\bar{d}}

\newcommand{\xx}{\mathbf{x}}
\newcommand{\XX}{\mathbf{X}}

\newcommand{\expn}{\mathrm{Exp}}
\newcommand{\cN}{\mathcal{N}}
\newcommand{\cF}{\mathcal{F}}
\newcommand{\cK}{\mathcal{K}}
\newcommand{\cM}{\mathcal{M}}
\newcommand{\cC}{\mathcal{C}}
\newcommand{\cX}{\mathcal{X}}
\newcommand{\htheta}{\hat{\theta}}
\newcommand{\ha}{\hat{a}}

\newcommand{\dif}{\ensuremath{\mbox{d}}}  
\newcommand{\e}{\mathrm{e}}

\newcommand{\pto}{\ensuremath{\xrightarrow{\mathbbm{P}}}}  

\newcommand{\expt}{\mathbb{E}}
\newcommand\ind[1]{\ensuremath{\mathbbm{1}_{\left[#1\right]}}} 
\newcommand\Pro[1]{\mathbbm{P}\left(#1\right)}  
\newcommand{\prob}{\mathbb{P}}

\newcommand\indn[1]{\ensuremath{\mathbbm{1}_{#1}}} 

\newtheorem{theorem}{Theorem}

\newtheorem{lemma}[theorem]{Lemma}
\newtheorem{proposition}[theorem]{Proposition}
\newtheorem{remark}{Remark}
\newtheorem{claim}{Claim}

\newtheorem{implication}{Implication}

\let\plainqed\qedsymbol
\newcommand{\claimqed}{$\lrcorner$}

\hypersetup{
 colorlinks,
 linkcolor=red,          
 citecolor=blue,       
 filecolor=red,   
 urlcolor=black, 
 pdftitle={},
 pdfauthor={},
 pdfcreator={Debankur Mukherjee},
 pdfsubject={},
 pdfkeywords={}
}

\usepackage{fancyhdr}
 
\tikzstyle{mybox} = [draw=red, fill=yellow!20, thick, minimum height=.4cm,
    rectangle, rounded corners]
\tikzstyle{fancytitle} =[fill=blue, text=white]
\usetikzlibrary{mindmap,trees,shadows}

\numberwithin{equation}{section}
\numberwithin{theorem}{section}

\tikzstyle{mybox} = [draw=black, thick, minimum height=.6cm,
    rectangle,text centered]
\tikzstyle{fancytitle} =[fill=blue, text=white]

\begin{document}

\title{Join-Idle-Queue with Service Elasticity: \\Large-Scale Asymptotics of a Non-monotone System}

\author[1]{Debankur Mukherjee\footnote{Email: \texttt{d.mukherjee@tue.nl}}}
\author[2]{Alexander Stolyar}
\affil[1]{Eindhoven University of Technology, The Netherlands}
\affil[2]{University of Illinois at Urbana-Champaign}

\renewcommand\Authands{, }

\date{\today}

\maketitle 

\begin{abstract}
We consider the model of a token-based joint auto-scaling and load balancing strategy, proposed in a recent paper by Mukherjee, Dhara,
 Borst, and van Leeuwaarden~\cite{MDBL17} (SIGMETRICS~'17), which offers an efficient scalable implementation and yet achieves asymptotically optimal steady-state delay performance and energy consumption as the number of servers $N\to\infty$.
In the above work, the asymptotic results are obtained \emph{under the assumption that the queues have fixed-size finite buffers}, and therefore the fundamental question of stability of the proposed scheme with infinite buffers was left open.
In this paper, we address this fundamental stability question.
The system stability under the usual subcritical load assumption is not automatic.
Moreover, the stability may \emph{ not} even hold for all $N$. 
The key challenge stems from the fact that the process \emph{lacks monotonicity}, which has been the powerful primary tool for establishing stability in load balancing models. 
We develop a novel method to prove that the subcritically loaded system is stable for \emph{large enough}~$N$, and establish convergence of steady-state distributions to the optimal one, as $N \to \infty$. 
The method goes beyond the state of the art techniques -- it uses an induction-based idea and a ``weak monotonicity'' property of the model; this technique is of independent interest and may have broader applicability.
\end{abstract}

\section{Introduction}
\noindent
{\bf Background and motivation.}
Load balancing and auto-scaling are two principal pillars in modern-day data centers and cloud networks, and therefore, have gained renewed interest in past two decades. 
In its basic setup, a large-scale system consists of a pool of large number of servers and a single dispatcher, where tasks arrive sequentially. 
Each task has to be instantaneously assigned to some server or discarded. 
Load balancing algorithms primarily concern design and analysis of algorithms to distribute incoming tasks among the servers as evenly as possible, while using minimal instantaneous queue length information.
At the same time, a big proportion of the tasks processed by these data centers come with  business-critical performance requirements.
This forces service providers to increase their capacity at a tremendous rate to cope up with the high-demand period in the presence of a time-varying demand pattern.
Consequently, the energy consumption by the servers in these huge data centers has risen dramatically and become a dominant factor in managing data center operations
and cloud infrastructure platforms. 
Auto-scaling provides a popular paradigm for automatically adjusting service capacity in response to demand while meeting performance targets.\\

\noindent
\emph{Load balancing in large systems.} The study of load balancing schemes in large-scale systems have a very rich history, and for decades, a lot of research have been conducted in understanding the fundamental trade-off between delay-performance and communication overhead per task.
The so-called `power-of-d' schemes, where each arrival is assigned to the shortest among $d$ randomly chosen queues, provide surprising improvements for $d\geq 2$ over purely random routing ($d=1$) while maintaining the communication overhead as low as~$d$ per task.
This scheme along with its many variations have been studied extensively in \cite{VDK96,Mitzenmacher01,BLP13,BLP12,AR17, EG16, MBLW16-3, Ying17} and many more.
Relatively recently, join-the-idle queue (JIQ) scheme was proposed in~\cite{LXKGLG11}, where an arriving task is assigned to an idle server (if any), or in case all servers are busy, it is assigned to some queue uniformly at random.
The JIQ scheme has a low-cost token-based implementation that involves only $O(1)$ communication overhead per task.
Large-scale asymptotic results in \cite{Stolyar15,Stolyar17} show that
under Markovian assumptions, the JIQ policy achieves a zero
probability of wait for any fixed subcritical load per server
in a regime where the total number of servers grows large.
It should be noted that the results in~\cite{Stolyar15,Stolyar17} even hold for considerably more general scenarios, viz.~decreasing hazard rate service time distributions and heterogeneous servers pools.
Recently, it is further shown~\cite{FS17} that when the average load per server $\lambda<1/2$, the  large-scale asymptotic optimality of JIQ is preserved even under completely general service time distributions.
Results in~\cite{MBLW16-1} indicate that under Markovian assumptions, the JIQ policy has the
same diffusion-limit as the Join-the-Shortest-Queue (JSQ)
strategy, and thus achieves diffusion-level optimality.
These results show that the JIQ policy provides asymptotically
optimal delay performance in large-scale systems, while only involving minimal
communication overhead (at most one message per task on average).
We refer to~\cite{BBLM18} for a recent survey on load balancing schemes.\\

\noindent
\emph{Auto-scaling with a centralized queue.}
Queue-driven auto-scaling techniques have been widely
investigated in the literature~\cite{ALW10, GDHS13, LCBWGWMH12, LLWLA11a, LLWLA11b, LLWA12, LWAT13, PP16, UKIN10, WLT12}.
In systems with a centralized queue it is very common to put servers to `sleep' while the demand is low, since servers in sleep mode consume much less energy than active servers.
Under Markovian assumptions, the behavior of these mechanisms can
be described in terms of various incarnations of M/M/N queues
with setup times.
There are several further recent papers which examine on-demand server addition/removal in a somewhat different vein~\cite{PS16, NS16}. 
 Generalizations towards non-stationary arrivals and impatience effects have also been considered recently~\cite{PP16}.
Unfortunately, data centers and cloud networks with millions of servers are too complex to maintain any centralized queue, and it involves prohibitively high communication burden to obtain instantaneous system information even for a small fraction of servers.\\

\noindent
\emph{Joint load balancing and auto-scaling in distributed systems.}
Motivated by all the above, a token-based joint load balancing and auto scaling scheme called TABS was proposed in~\cite{MDBL17}, that offers an efficient scalable implementation and yet achieves asymptotically optimal steady-state delay performance and energy consumption as the number of servers $N\to\infty$.
In~\cite{MDBL17}, the authors left open a fundamental question: Is the system with a given number $N$ of servers stable under TABS scheme?
The analysis in~\cite{MDBL17} bypasses the issue of stability by assuming that each server in the system has a finite buffer capacity.
Thus, it remains an important open challenge to understand the stability property of the TABS scheme without the finite-buffer restriction. \\

\noindent
{\bf Key contributions and our approach.}
In this paper we address the stability issue for systems under the TABS scheme without the assumption of finite buffers, and examine the asymptotic behavior of the system as $N$ becomes large.
Analyzing the stability of the TABS scheme in the infinite buffer scenario poses a significant challenge, because the stability of the finite-$N$ system, i.e., the system with finite number $N$ of servers under the usual subcritical load assumption is not automatic.
In fact, as we will further discuss in Remark~\ref{rem:instability} below in detail, even under subcritical load, the system may \emph{not} be stable for all $N$.
Our first main result is that for any fixed subcritical load, the system is stable for \emph{large enough} $N$.
Further, using this large-$N$ stability result in combination with mean-field analysis, we establish convergence of the sequence of steady-state distributions as~$N\to\infty$.

The key challenge in showing large-$N$ stability for systems under the TABS scheme stems from the fact that the occupancy state process lacks monotonicity.
It is well-known that monotonicity is a powerful primary tool for establishing stability of load balancing models~\cite{Stolyar15,Stolyar17, VDK96, BLP12}.
In fact, process monotonicity is used extensively not only for stability analysis and not only in queueing literature -- for example,
many interacting-particle-systems' results rely crucially on monotonicity; see e.g.~\cite{L85}. The lack of monotonicity immediately complicates the situation, as for example in ~\cite{FS17, SS17}.
Specifically, when the service time distribution is general, it is the lack of monotonicity that has left open the stability questions for the power-of-d scheme when system load $\lambda>1/4$~\cite{BLP12}, and the JIQ scheme when $\lambda>1/2$~\cite{FS17}.
We develop a novel method for proving \emph{large-$N$ stability} for subcritically loaded systems, and using that we establish the convergence of the sequence of steady-state distributions as $N\to\infty$.
Our method uses an induction-based idea, and relies on a ``weak monotonicity'' property of the model, as further detailed below.
To the best of our knowledge, this is the first time both the \emph{traditional fluid limit} (in the sense of large starting state) and the \emph{mean-field fluid limit} (when the number of servers grows large) are used in an intricate manner to obtain large-$N$ stability results.

To establish the large-$N$ stability, we actually prove a stronger statement.
We consider an artificial system, where some of the queues are infinite at all times. 
Then, loosely speaking, we prove that the following holds for all sufficiently large $N$:
\emph{If the system with $N$ servers contains $k$ servers with infinite queue lengths, $0\leq k\leq N$, then (i)~The subsystem consisting of the remaining (i.e., finite) queues is stable, and 
(ii)~When this subsystem is at steady state, the average rate at which tasks join the infinite queues is strictly smaller than that at which tasks depart from them.}
Note that the case $k=0$ corresponds to the desired stability result.

The use of backward induction in $k$ facilitates proving the above statement.
For a fixed $N$, first we introduce the notion of a fluid sample path (FSP) for systems where some queues might be infinite.
The base case of the backward induction is when $k=N$, and assuming the statement for $k$, we show that it holds for $k-1$.
We use the classical fluid-stability argument (as in~\cite{RS92, S95, D99}) in order to establish stability for the system where the number of infinite queues is $k-1$.
As mentioned above, here the notion of the traditional FSP is needed to be suitably extended to fit to the systems where some servers have infinite queue lengths.
Loosely speaking, for the fluid-stability, the `large queues' behave as `infinite queues' for which induction statement provides us with the drift estimates.
Also, to calculate the drift of a queue in the fluid limit for fixed but \emph{large enough $N$}, we use the mean-field analysis.
A more detailed heuristic roadmap of the above proof argument is presented in Subsection~\ref{ssec:roadmap}.
This technique is of independent interest, and potentially has a much broader applicability in proving large-$N$ stability for non-monotone systems, where the state-of-the-art results have remained scarce so far.
\\


\noindent
{\bf Organization of the paper.}
The rest of the paper is organized as follows.
In Section~\ref{sec:model} we present a detailed model description,  state the main results, and provide their ramifications along with discussions of several proof heuristics.
The full proof of the main results are deferred till Section~\ref{sec:proofs}. 
Section~\ref{sec:ind} introduces an inductive approach to prove large-$N$ stability result.
We present the proof of the large-scale system (when $N\to\infty$) using mean-field analysis in Section~\ref{sec:mf}.
Finally, we make a few brief concluding remarks in Section~\ref{sec:con}.

\section{Model description and main result}\label{sec:model}

In this section, first we will describe the system and the TABS scheme in detail, and then state the main results and discuss their ramifications.

Consider a system of $N$~parallel queues with identical servers and a single dispatcher. 
Tasks with unit-mean exponentially distributed service requirements arrive as a Poisson process of rate $\lambda N$  with $\lambda<1$.
Incoming tasks cannot be queued at the dispatcher, and must immediately and irrevocably be forwarded to one of the servers where they can be queued.
Each server has an infinite buffer capacity. 
The service discipline at each server is oblivious to the actual service requirements (e.g., FCFS).
A turned-off server takes an Exponentially distributed time with mean $1/\nu$ (to be henceforth denoted as Exp$(\nu)$) time (setup period) to be turned on.
We now describe the token-based joint auto-scaling and load balancing scheme called TABS (Token-based Auto Balance Scaling), as introduced in~\cite{MDBL17}.\\

\noindent
{\bf TABS scheme~\cite{MDBL17}.}
\begin{itemize}
\item When a server becomes idle, it sends a `green' message to the dispatcher, waits for an $\expn(\mu)$ time (standby period), and turns itself off by sending a `red' message to the dispatcher (the corresponding green message is destroyed).
\item When a task arrives, the dispatcher selects a green message at random if there are any, and assigns the task to the corresponding server (the corresponding green message is replaced by a `yellow' message). 
Otherwise, the task is assigned to an arbitrary busy server (and is lost if there is none), and if at that arrival epoch there is a red message at the dispatcher, then it selects one at random, and the setup procedure of the corresponding server is initiated, replacing its red message by an `orange' message.
Setup procedure takes $\expn(\nu)$ time after which the server becomes active. 
\item Any server which activates due to the latter event, sends a green message to the dispatcher (the corresponding orange message is replaced), waits for an $\expn(\mu)$ time for a possible assignment of a task, and again turns itself off by sending a red message to the dispatcher.
\end{itemize}
\begin{figure}
\begin{center}
\includegraphics[scale=1]{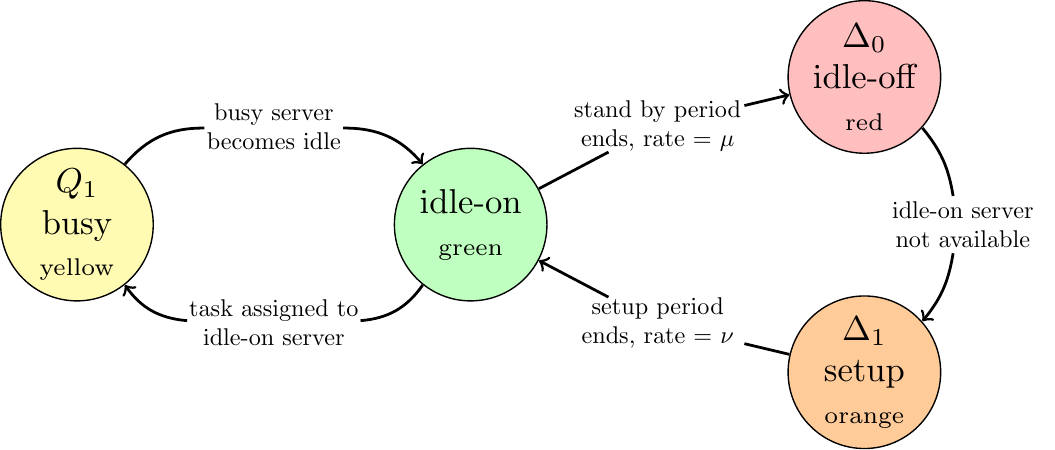}
\end{center}
\caption{Illustration of server on-off decision rules in the TABS scheme, along with message colors and state variables as given in~\cite{MDBL17}.}
\label{fig:scheme}
\end{figure}
As described in~\cite{MDBL17}, the TABS scheme gives rise to a distributed operation in which  servers are in one of four states (busy, idle-on, idle-off, or standby), and advertize their state to the dispatcher via exchange of tokens. Figure~\ref{fig:scheme} illustrates this token-based exchange protocol. 
Note that setup procedures are never aborted and continued even when idle-on servers do become available.  \\

\noindent
{\bf Notation.}
For the system with $N$ servers, let $X_j^N(t)$ denote the queue length of server $j$ at time~$t$, $j=1,2,\ldots, N$, and
 $\mathbf{Q}^N(t) := (Q_1^N(t), Q_2^N(t), \dots )$ denote the system occupancy state, where $Q_i^N(t)$ is the number of servers with queue length greater than or equal to $i$ at time $t$, including the possible task in service, $i=1,2,\ldots$.
Also, let $\Delta_0^N(t)$ and $\Delta_1^N(t)$ denote the number of idle-off servers and servers  in setup mode at time $t$, respectively. 
Note that the process $(\mathbf{Q}^N(t),\Delta_0^N(t),\Delta_1^N(t))_{t\geq 0}$ provides a Markovian state description by virtue of the exchangeablity of the servers.
It is easy to see that, for any fixed $N$, this process is an irreducible countable-state Markov chain. 
Therefore, its positive recurrence, which we refer to as \emph{stability}, is equivalent to ergodicity and to the existence of unique stationary distribution. 
Further, let  $U^N(t)$ denote the number of idle-on servers at time $t$.
We will focus upon an asymptotic analysis, where the task arrival rate and the number of servers grow large in proportion.
The \emph{mean-field fluid-scaled} quantities are denoted by the respective small letters, viz.~$q_i^{N}(t):=Q_i^{N}(t)/N$, $\delta_0^N(t) = \Delta_0^N(t)/N$, $\delta_1^N(t) = \Delta_1^N(t)/N$, and $u^N(t):= U^N(t)/N$. 
Notation for the \emph{conventional fluid-scaled} occupancy states for a fixed~$N$ will be introduced later in Subsection~\ref{ssec:fluid}.
For brevity in notation, we will write $\mathbf{q}^N(t) = (q_1^N(t), q_2^N(t), \dots)$ and $\bld{\delta}^N(t) = (\delta_0^N(t),\delta_1^N(t))$. 
Let
$$
E = \Big\{(\bld{q},\dd)\in [0,1]^\infty:  q_i\geq q_{i+1},\ \forall i, \  \delta_0+\delta_1+ q_1\leq 1 \Big\},
$$
denote the space of all mean-field fluid-scaled occupancy states,
so that  $(\mathbf{q}^N(t),\bld{\delta}^N(t))$ takes value in $E$ for all $t$.
Endow $E$ with the product topology, and the Borel $\sigma$-algebra $\mathcal{E}$, generated by the open sets of $E$.
For any complete separable metric space $E$, denote by $D_E[0,\infty)$, the set of all $E$-valued \emph{c\`adl\`ag} (right continuous with left limit exists) processes.
By the symbol `$\pto$' we denote convergence in probability for real-valued random variables.\\

We now present our first main result which states that for any fixed choice of the parameters, a sub-critically loaded system under TABS scheme is stable for large enough $N$. 
\begin{theorem}\label{th:stab}
For any fixed $\mu$, $\nu>0,$ and $\lambda<1$, the system with $N$ servers under the TABS scheme is stable (positive recurrent) for large enough $N$. 
\end{theorem}
Theorem~\ref{th:stab} is proved in Section~\ref{sec:proofs}.

\begin{remark}\label{rem:instability}\normalfont
It is worthwhile to mention that the `large-$N$' stability as stated in Theorem~\ref{th:stab} above is the best one can hope for. 
In fact, for fixed $N$ and $\lambda$, there are values of the parameters $\mu$ and $\nu$ such that the system under the TABS scheme may not be stable.
To elaborate further on this point, consider a system with 2 servers A and B, and $1/2<\lambda<1$.
Let server A start with a large queue, while the initial queue length at server B be small.
In that case, observe that every time the queue length at server B hits 0, with positive probability, it turns idle-off before the next arrival epoch.
Once server B is idle-off, the arrival rate into server A becomes $2\lambda>1$. 
Thus, before server B turns idle-on again, the expected number of tasks that join server A is given by at least $2\lambda/\nu$, while the expected number of departures is $1/\nu$.
Thus the queue length at server A increases by $(2\lambda-1)/\nu$,
which can be very large if $\nu$ is small.
Further note that once server B becomes busy again, both servers receive an arrival rate $\lambda<1$, and hence it is more likely that server B will empty out again, repeating the above scenario.
The situation becomes better as $N$ increases. 
Indeed for large $N$, if `too many' servers are idle-off and `too many' tasks do not find an idle queue to join, the system starts producing servers in setup mode fast enough, and as a result, more and more servers start becoming busy.
The above heuristic has been illustrated in Figure~\ref{fig:instability} with examples of three scenarios with small, moderate, and large values of $N$, respectively.
\begin{figure}
\begin{center}$
\begin{array}{ccc}
\includegraphics[width=75mm]{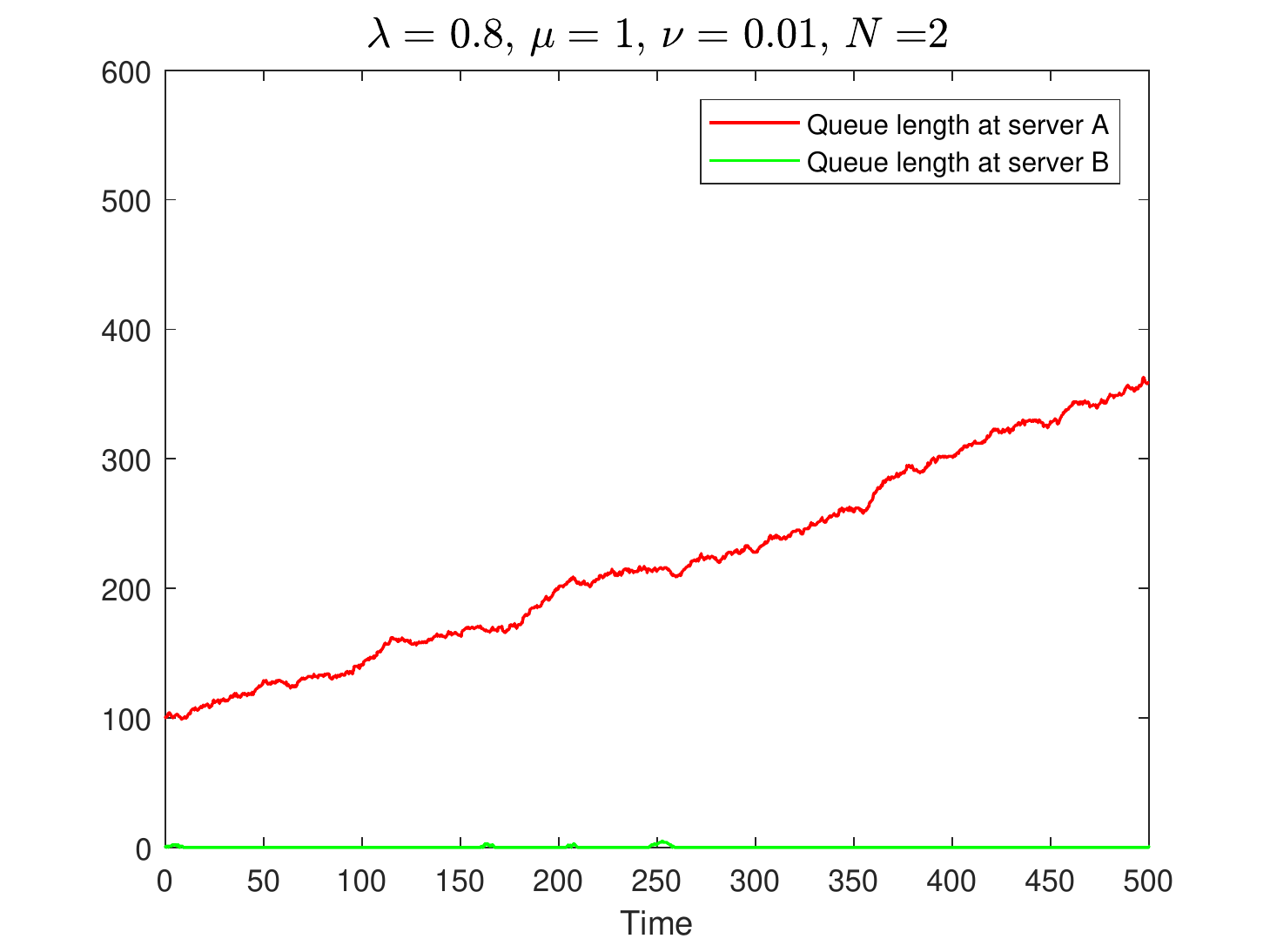}
\includegraphics[width=75mm]{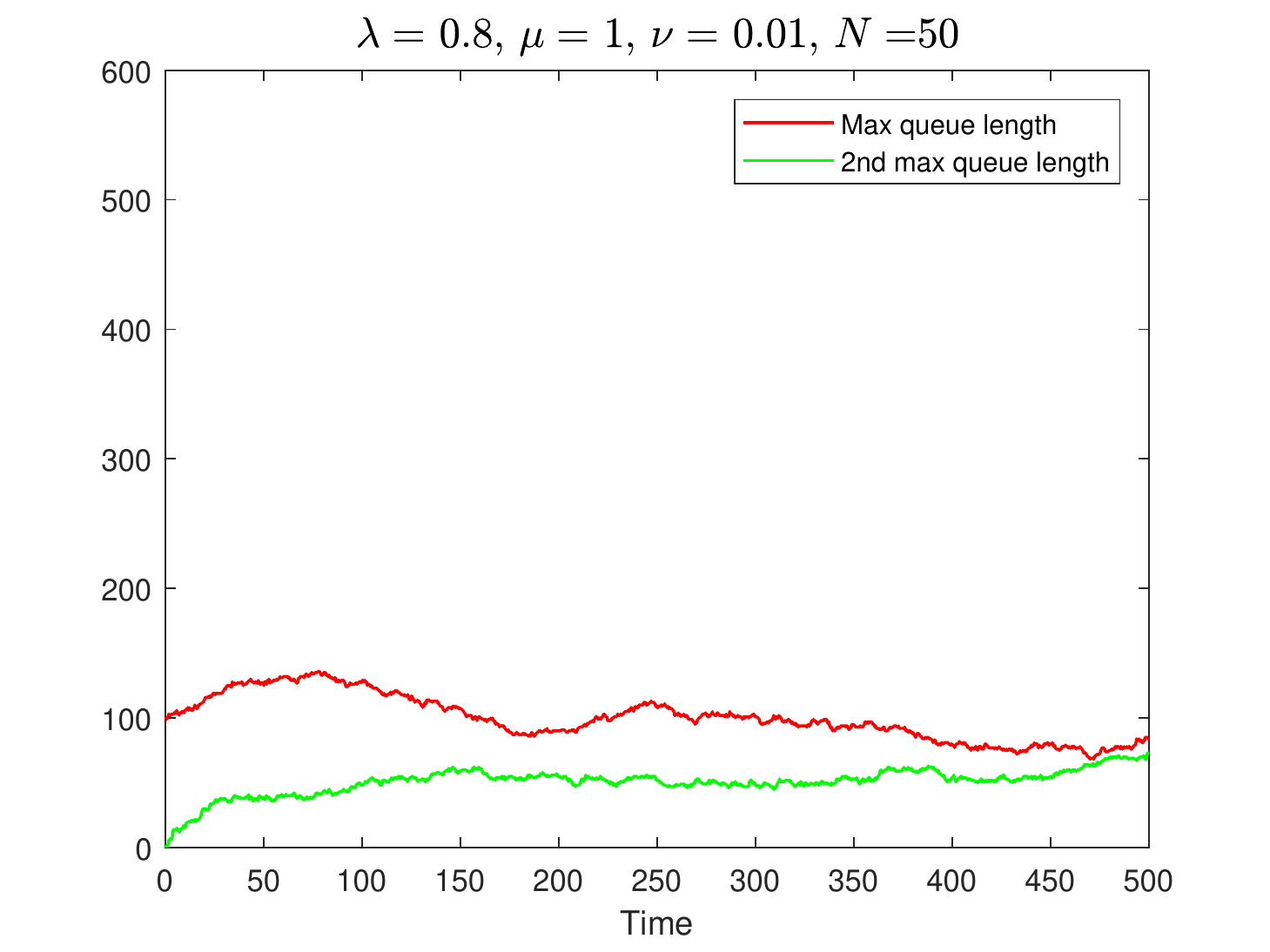}\\
\includegraphics[width=75mm]{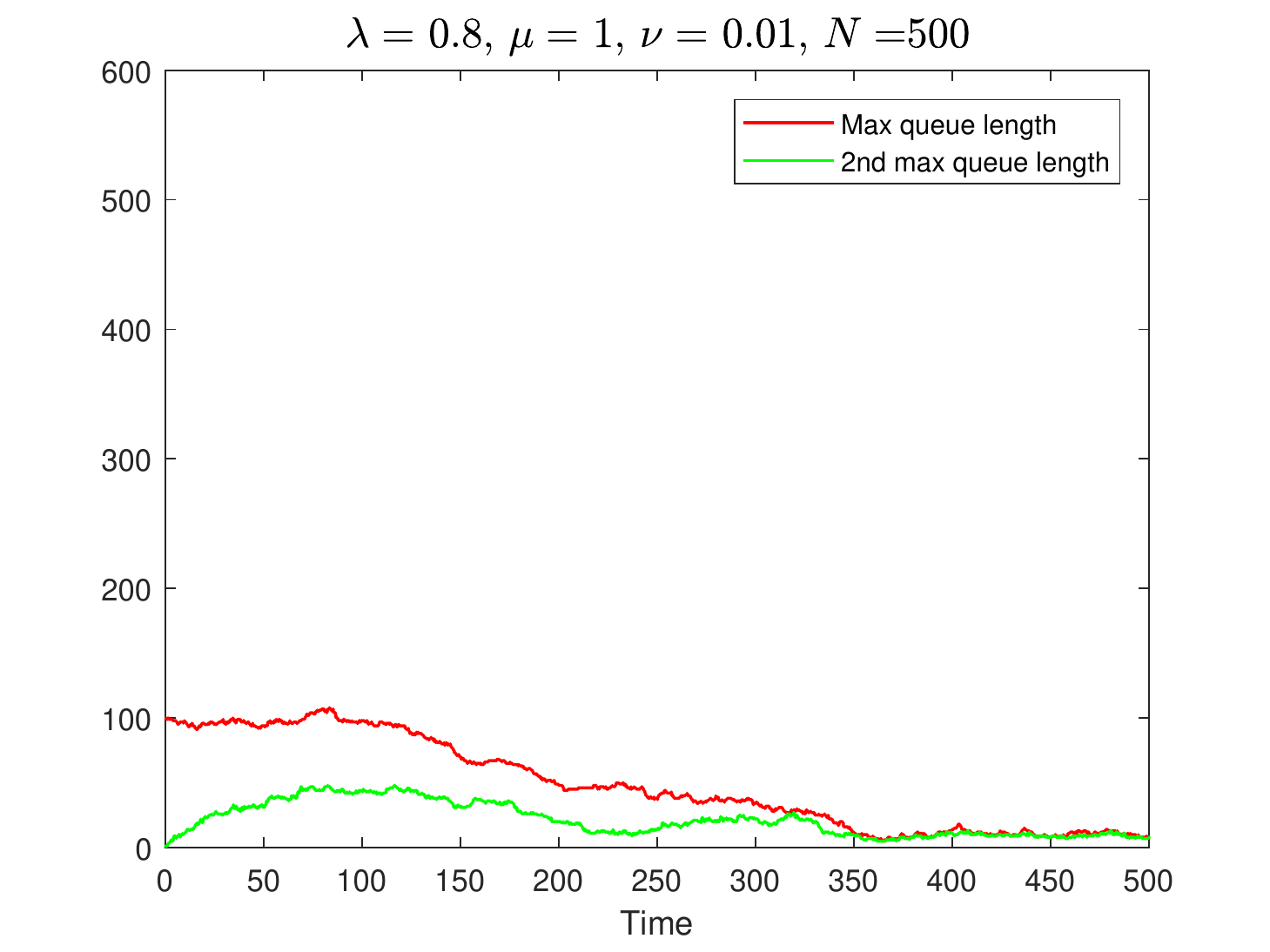}
\end{array}$
\end{center}
\caption{(Top left) Illustration of instability of the TABS scheme for $N=2$ via sample paths of the queue length process. (Top right) Sample paths of the maximum and second maximum queue length processes in an intermediate system ($N=50$) for the same parameter choices. (Bottom) The system becomes stable for a large enough system ($N=500$).}
\label{fig:instability}
\end{figure}
\end{remark}

In the next theorem we will identify the limit of the sequence of stationary distributions of the occupancy processes as $N\to\infty$.
In particular, we will establish that under sub-critical load, for any fixed $\mu$, $\nu>0$, the steady-state occupancy process converges weakly to the unique fixed point.
(For the finite buffer scenario this was proved in~\cite[Proposition 3.3]{MDBL17}.)
Denote by $\qq^N(\infty)$ and $\dd^N(\infty)$ the random values of $\qq^N(t)$ and $\dd^N(t)$ in the steady-state, respectively.
\begin{theorem}\label{th:steady-lim}
For any fixed $\mu$, $\nu>0$, and $\lambda<1$, the sequence of steady states
$(\qq^N(\infty), \dd^N(\infty))$ converges weakly to the fixed point $(\qq^\star, \dd^\star)$  as  $N\to\infty$,
where 
\[\delta_0^\star = 1-\lambda\quad\delta_1^\star = 0\quad q_1^\star = \lambda,\quad q_i^\star =0 \quad\mbox{for all}\quad i\geq 2.\]
\end{theorem}
Note that the fixed point $(\qq^\star, \dd^\star)$ is such that the  probability of wait vanishes as $N\to\infty$ and the asymptotic fraction of active servers is minimum possible, and in this sense, the fixed point is optimal.
Thus, Theorem~\ref{th:steady-lim} implies that the TABS scheme provides fluid-level optimality for large-scale systems in terms of delay performance and resource utilization,
while involving only $O(1)$ communication overhead per task.

\section{Proofs of the main results}\label{sec:proofs}

In Subsection~\ref{ssec:fluid} we introduce 
the notion of conventional fluid scaling (when the number of servers is fixed) and fluid sample paths (FSP), and state Proposition~\ref{prop:largeN} that implies Theorem~\ref{th:stab} as an immediate corollary.
Subsection~\ref{ssec:large-scale} contains two key results for sequence of systems with increasing system size, i.e., number of servers $N\to\infty$, and proves Theorem~\ref{th:steady-lim}.

\subsection{Conventional fluid limit for a system with fixed N}\label{ssec:fluid}
In this subsection first we will introduce a notion of fluid sample path (FSP) for finite-$N$ systems where some of the queue lengths are infinite.
We emphasize that this is \emph{conventional fluid limit}, in the sense that the number of servers is fixed, but the time and the queue length at each server are scaled by some parameter that goes to infinity.

Loosely speaking, conventional fluid limits are usually defined as follows: 
For a fixed $N$, consider a sequence of systems with increasing initial norm (total queue length) $R$ say. 
Now scale the queue length process at each server and the time by $R$.
Then any weak limit of this sequence of (space and time) scaled processes is called an FSP.
Observe that this definition is inherently not fit if the system has some servers whose initial queue length is infinity.
Thus we introduce a suitable notion of FSP that does not require the scaled norm of the initial state to be 1.
We now introduce a rigorous notion of FSP for systems with some of the queues being infinite.\\

\noindent
{\bf Fluid limit of a system with some of the queues being infinite.} 
Consider a system of $N$ servers with indices in $\cN$, among which $k$ servers with indices in $\cK\subseteq\cN$ have infinite queue lengths.
Now consider any sequence of systems indexed by $R$ such that $\sum_{i\in\cN\setminus\cK}X_i^{N,R}(0) <\infty$, and 
\begin{equation}\label{eq:conventional}
x^{N,R}_i(t) := \frac{X_i^{N,R}(Rt)}{R},\quad i\in\cN\setminus\cK
\end{equation}
be the corresponding scaled processes.
For fixed $N$, the scaling in~\eqref{eq:conventional} will henceforth be called as the \emph{conventional fluid-scaled} queue length process.
Also, for the $R$-th system, let $A_i^{N,R}(t)$ and $D_i^{N,R}(t)$ denote the cumulative number of arrivals to and departures from server $i$ with $a_i^{N,R}(t):=A_i^{N,R}(Rt)/R$ and $d_i^{N,R}(t):=D_i^{N,R}(Rt)/R$ being the corresponding fluid-scaled processes, $i\in\cN$.
We will often omit the superscript $N$ when it is fixed from the context.

Now for any fixed $N$, suppose the (conventional fluid-scaled) initial states converge, i.e., $x^{R}(0) \to x(0)$, for some fixed
$x(0)$ such that $0 \le \sum_{i\in\cN\setminus\cK} x_i(0) < \infty$ and $x_i(0) = \infty$ for $i\in\cK$.
Then a set of uniformly Lipschitz continuous functions $(x_i(t), a_i(t), d_i(t))_{i\in\cN}$ on the time interval $[0,T]$ (where $T$ is possibly infinite) with the convention $x_i(\cdot)\equiv \infty$ for all $i\in \cK$, is called a \emph{fluid sample path} (FSP) starting from $\xx(0)$, if for any subsequence of $\{R\}$ there exists a further subsequence (which we still denote by $\{R\}$) such that with probability~1, along that subsequence the following convergences hold:
\begin{enumerate}[{\normalfont (i)}]
\item For all $i\in \cN$, $a_i^{R}(\cdot)\to a_i(\cdot)$ and $d_i^{R}(\cdot)\to d_i(\cdot)$, u.o.c.
\item For $i\in\cN\setminus\cK$, $x_i^{R}(\cdot)\to x_i(\cdot)$ u.o.c.
\end{enumerate}
Note that the above definition is equivalent to convergence in probability to the unique FSP.
For any FSP almost all points (with respect to Lebesgue measure) are \emph{regular}, i.e., for all $i\in \cN\setminus\cK$, $x_i(t)$ has proper left and right derivatives with respect to $t$, and for all such regular points \[x_i'(t) = a_i'(t)-d_i'(t).\]

\noindent
{\bf Infinite queues as part of an FSP.} The arrival and departure functions $a_i(t)$ and $d_i(t)$ are well-defined for each queue, including infinite queues. Of course, the derivative $x'_i(t)$ for an infinite queue makes no direct sense (because an infinite queue remains infinite at all times). However, we adopt a convention that $x'_i(t)=a'_i(t)-d'_i(t)$, for all queues, including the infinite ones. For an FSP, $x'_i(t)$ is sometimes referred to as a ``drift'' of (finite or infinite) queue $i$ at time $t$.\\

We are now in a position to state the key result that establishes the large-$N$ stability of the TABS scheme.
\begin{proposition}\label{prop:largeN}
The following holds for all sufficiently large $N$. 
For each $0\leq k\leq N$, consider a system where $k$ servers with indices in $\cK$ have infinite queues, and the remaining $N-k$ queues are finite.
Then, for each $j=1,2,\ldots,N$, there exists $\varepsilon(j) >0$, such that the following properties hold ($\varepsilon(j)$ and other constants specified below, also depend on $N$).
\begin{enumerate}[{\normalfont (1)}]
\item For any $\xx(0)$ such that $0 \le \sum_{i\in\cN\setminus\cK} x_i(0) < \infty$ and $x_i(0) = \infty$ for $i\in\cK$,
there exists $T(k,\xx(0))<\infty$ and a unique FSP on the interval $[0,T(k,\xx(0))]$, which has the following properties:
	\begin{enumerate}[{\normalfont (i)}]
	\item If at a regular point $t$, $\cM(t):=\{i\in \cN: x_i(t)>0\}$ with $|\cM(t)|=m>k$, then $x_i'(t) = -\varepsilon(m)$ for all $i\in \cM(t)$.
	\item For any $i\in\cN\setminus\cK$, if $x_i(t_0)=0$ for some $t_0$, then $x_i(t)=0$ for all $t\geq t_0$.
	\item $T(k,\xx(0))=\inf\ \{t:x_i(t) = 0\mbox{ for all }i\in\cN\setminus\cK\}$.
	\end{enumerate}
\item The subsystem with $N-k$ finite queues is stable.
\item When the subsystem with $N-k$ finite queues is in steady state, the average arrival rate into each of the $k$ servers  having infinite queue lengths is at most $1-\varepsilon(k)$.
\item For any $x(0)$ such that $0 \le \sum_{i\in\cN\setminus\cK} x_i(0) < \infty$ and $x_i(0) = \infty$ for $i\in\cK$,
there exists a unique FSP on the entire interval $[0,\infty)$. 
In $[0,T(k,x(0))]$, it is as described in Statement 1.
Starting from $T(k,x(0))$, all queues in $\cN\setminus\cK$
stay at $0$ and all infinite queues have drift at most $-\varepsilon(k)$.
\end{enumerate}
\end{proposition}
Although Part 2 follows from Part 1, and Part 4 is stronger than Part 1, the statement of Proposition~\ref{prop:largeN} is arranged as it is to facilitate its proof, as we will see in Section~\ref{sec:ind} in detail.
\begin{proof}[Proof of Theorem~\ref{th:stab}]
Note that Theorem~\ref{th:stab} is a special case of Proposition~\ref{prop:largeN} when $k=0$.
\end{proof}

\subsection{Large-scale asymptotics: auxiliary results}
\label{ssec:large-scale}
In this subsection we will state two crucial lemmas that describe asymptotic properties of sequence of systems as the number of servers $N\to\infty$, \emph{if stability is given}.
Their proofs involve mean-field fluid scaling and limits.
\begin{lemma}\label{lem:expo-bound}
There exist $\varepsilon_1>0$ and $C_q=C_q(\varepsilon_1)>0$, such that the following holds.
Consider any sequence of systems with $N\to\infty$ and $k=k(N)$ infinite queues such that $k(N)/N \to \kappa\in [0,1]$, and assume that each of these systems is stable.
Then for all sufficiently large $N$,
\[\Pro{q_1^N(\infty)<\varepsilon_1}\leq \e^{-C_qN}.\]
\end{lemma}
\begin{lemma}\label{lem:steady-concentration}
Consider any sequence of systems with $N\to\infty$ and $k=k(N)$ infinite queues such that $k(N)/N \to \kappa\in [0,1]$, and assume that each of these systems is stable.
The following statements hold:
\begin{enumerate}[{\normalfont (1)}]
\item If $\kappa \geq 1-\lambda$, then $q_1^N(\infty)\pto 1$ as $N\to\infty$. 
\item If $\kappa < 1-\lambda$, then 
the limit of the sequence of stationary occupancy states $(\qq^N(\infty),\dd^N(\infty))$ is the distribution concentrated at the unique equilibrium point $(\qq^\star(\kappa), \dd^\star(\kappa))$, such that
\begin{align*}
q_1^\star(\kappa) &= \kappa+\lambda,\quad q_2^\star(\kappa) = \kappa\\
\delta_0^\star(\kappa) &= 1-\lambda-\kappa,\quad\delta_1^\star(\kappa) = 0.
\end{align*}
Consequently,
\begin{equation}\label{eq:foundbusy}
\lim_{N\to\infty}\Pro{Q_1^N(\infty)+\Delta_0^N(\infty)+\Delta_1^N(\infty)=N}=0.
\end{equation}
\end{enumerate}
\end{lemma} 
Lemmas~\ref{lem:expo-bound} and \ref{lem:steady-concentration} are proved in Section~\ref{sec:mf}. 
These results will be used to derive necessary large-$N$ bounds on the expected arrival rate into each of the servers having infinite queue lengths when the system is in steady state.
\begin{remark}\normalfont
It is also worthwhile to note that Lemmas~\ref{lem:expo-bound} and~\ref{lem:steady-concentration} can be thought of as a \emph{weak monotonicity} property of the TABS scheme as mentioned earlier.
Loosely speaking, the weak monotonicity requires that no matter where the system starts, in some fixed time the system arrives at a state with a certain fraction of busy servers.
The purpose of Lemmas~\ref{lem:expo-bound} and~\ref{lem:steady-concentration} is to bound \emph{under the assumption of stability}, the expected rate at which task arrives to the infinite queues when the subsystem containing the finite queues is in steady-state: In this regard
\begin{enumerate}[{\normalfont (i)}]
\item Lemma~\ref{lem:steady-concentration} guarantees high probability bounds on the total number of busy servers, so that with probability tending to 1 as $N\to\infty$, the fraction of busy servers in the whole system is at least $\lambda$ in steady state.
\item But note that since the arrival rate is $\lambda N$, when the system has few busy servers (even with an asymptotically vanishing probability), the arrival rate to the infinite servers can become $\Theta(N)$.
Thus we need the exponential bound stated in Lemma~\ref{lem:expo-bound} in order to obtain bound on the expected rate of arrivals to the infinite queues.
\end{enumerate}
In Subsection~\ref{ssec:part3} we will see that as a consequence of Lemmas~\ref{lem:expo-bound} and~\ref{lem:steady-concentration}, we obtain that for large enough $N$, under the assumption of stability, the steady-state rate at which tasks join an infinite queue is strictly less than 1, and the drift of the infinite queues as defined in Subsection~\ref{ssec:fluid} becomes strictly negative.
This fact will be used in the proof of Proposition~\ref{prop:largeN}.
\end{remark}

\begin{proof}[Proof of Theorem~\ref{th:steady-lim}]
Note that given the large-$N$ stability property proved in Proposition~\ref{prop:largeN} for $k(N)=0$, and the convergence of 
stationary distributions under the assumption of stability in Lemma~\ref{lem:steady-concentration}, the proof of Theorem~\ref{th:steady-lim} is immediate.
\end{proof}

\section{Proof of \texorpdfstring{Proposition~\ref{prop:largeN}}{Proposition 3.1}: An inductive approach}\label{sec:ind}
Throughout this section we will prove Proposition~\ref{prop:largeN}.
The proof consists of several steps and uses both conventional fluid limit and mean-field fluid scaling and limit in an intricate fashion.
Below we first provide a roadmap of the whole proof argument.

\subsection{Proof idea and the roadmap}\label{ssec:roadmap}
The key idea for the proof of Proposition~\ref{prop:largeN} is to use backward induction in $k$, starting from the base case $k=N$. 
For $k=N$, all the queues are infinite.
In that case, Parts (1) and (2) are vacuously satisfied with the convention $T(N,\xx(0))=0$.
Further observe that TABS scheme does not differentiate between two large queues (in fact, any two non-empty queues).
Thus, when all queues are infinite, since all servers are always busy, each arriving task is assigned uniformly at random, and each server has an arrival rate $\lambda$ and a departure rate 1.
Thus, it is immediate that the drift of each server is $-(1-\lambda)<0$, and thus, $\varepsilon(N)=1-\lambda$.
This proves (3), and then (4) follows as well.

Now, we discuss the ideas to establish the backward induction step, i.e., assume that Parts (1)--(4) hold for $k\geq k(N)+1$ for some $k(N)\in \{0,1,\ldots, N-1\}$ and verify that the statements hold for $k=k(N)$.
Rigorous proofs to verify Parts (1)--(4) for $k=k(N)$ are presented  in Subsections~\ref{ssec:part1}--\ref{ssec:part4}.
We begin by providing a roadmap of these four subsections.

\paragraph{Part (1).} Recall that we denote by $\cK$ the indices of the servers having infinite queue lengths, and by $\cN$ the set of all server indices.
Denote by $x_{(i)}$ the $i$-th largest component of $\xx$ (ties are broken arbitrarily).
Then for any $\xx$ with $m\in  \{0,1,\ldots, N-1\}$ infinite components, define 
\begin{equation}\label{eq:Tkx}
T(m,\xx) := \frac{x_{(N)}}{\varepsilon(N)}+\sum_{i=1}^{N-m-1}\frac{x_{(N-i)}  -  x_{(N-i+1)}}{\varepsilon(N-i)}
\end{equation}
with the convention that $T(N,\xx)=0$ if all components of $\xx$ are infinite. 
For $k(N)\in \{0, 1,\ldots, N-1\}$, Part (1) is proved with the choice of $T(k, \xx(0))$ as given by~\eqref{eq:Tkx}. 
Indeed, recall that we are at the backward induction step where there are $k(N)$ infinite queues, and we also know from the hypothesis that Parts (1)--(4) hold if there are $k(N)+1$ or larger infinite queues in the system.
Loosely speaking, the idea is that as long as a conventional fluid-scaled queue length $x_j(t)$ at some server $j\in \cN\setminus\cK$ is positive, it can be coupled with a system where the queue length at server $j$ is infinite.
Thus, as long as there is at least one server $j\in \cN\setminus\cK$ with $x_j(t)>0$, the system can be `treated' as a system with at least $k(N)+1$ infinite queues, in which case, Part~(4) of the backward induction hypothesis furnishes with the drift of each positive component of the FSP (in turn, which is equal to the drift of each infinite queue for the corresponding system).

Now to explain the choice of $T(m,\xx)$ in~\eqref{eq:Tkx}, observe that when all the components of the $N$-dimensional FSP are strictly positive, each component has a negative drift of $-\varepsilon(N)$.
Thus, $x_{(N)}/\varepsilon(N)$ is the time when at least one component of the $N$-dimensional FSP hits 0.
From this time-point onwards, each positive component has a  drift of $-\varepsilon(N-1)$, and thus,  $x_{(N)}/\varepsilon(N) + (x_{(N-1)} - x_{(N)})/\varepsilon(N-1)$ is the time when two components hit 0.
Proceeding this way, one can see that at time $T(m, \xx(0))$ all finite positive components of the FSP hit 0.
The above argument is formalized in Subsection~\ref{ssec:part1}.

\paragraph{Part (1) $\boldsymbol{\implies}$ Part (2).}
To prove Part 2, we will use the fluid limit technique of proving stochastic stability as in~\cite{RS92, S95, D99}, see for example \cite[Theorem 4.2]{D99} or \cite[Theorem 7.2]{S95} for a rigorous statement.
Here we need to show that the sum of the non-infinite queues (of an FSP) drains to~0. This is true, because by Part~(1) each positive non-infinite queue will have negative drift. The formal proof is in Subsection~\ref{ssec:part2}.

\paragraph{Part (2) + Lemmas~\ref{lem:expo-bound} and~\ref{lem:steady-concentration} $\boldsymbol{\implies}$ Part (3).}
Note that in the proofs of Parts (1) and (2) we have only used the backward induction hypothesis, and have not imposed any restriction on the value of $N$.
This is the only part where in the proof we use the large-scale asymptotics, in particular, Lemmas~\ref{lem:expo-bound} and~\ref{lem:steady-concentration}.
For that reason, in the statement of Proposition~\ref{prop:largeN} we use ``large-enough $N$''.
The idea here is to prove by contradiction.
Suppose Part (3) does not hold for infinitely many values of $N$. 
In that case, it can be argued that there exist a subsequence $\{N\}$ and \emph{some} sequence $\{k(N)\}$ with $k(N)\in \{0,1,\ldots, N-1\}$, such that when the subsystem consisting of $N-k(N)$ finite queues is in the steady state, the average arrival rate into each of the $k(N)$ servers having infinite queue lengths is at least~1, along the subsequence.
Loosely speaking, in that case, Lemmas~\ref{lem:expo-bound} and~\ref{lem:steady-concentration} together imply that for large enough $N$, there are `enough' busy servers, so that the rate of arrival to each infinite queues is strictly smaller than~1, which leads to a contradiction.
Note that we can apply Lemmas~\ref{lem:expo-bound} and~\ref{lem:steady-concentration} here, because Part~(2) ensures the required stability. The rigorous proof is in Subsection~\ref{ssec:part3}.

\paragraph{Parts (2), (3) + Time-scale separation $\boldsymbol{\implies}$ Part (4).}
We assume that Parts (1) -- (3) hold for $k \in \{k(N), k(N)+1, \ldots, N\}$, and we will verify Part (4) for $k=k(N)$.
Observe that it only remains to prove convergence to the FSP on the (scaled) time interval $[T(k,\xx(0)),\infty]$.
For this, observe that it is enough to consider the sequence of systems for which $\xx^R(0)\to\xx(0)$ where $x_i(0)=0$ for all $i\in\cN\setminus\cK$. 
In particular, all that remains to be shown is that the drift of each infinite queue is indeed~$-\varepsilon(k)$.
Recall the conventional fluid scaling and FSP from Subsection~\ref{ssec:fluid}, and let $R$ be the scaling parameter.
The proof consists of two main parts:
\begin{enumerate}[{\normalfont (i)}]
\item Let us fix any state $z$ of the \emph{unscaled} process. If the sequence of systems is such that $\xx^R(0)\to\xx(0)$ where $x_i(0)=0$ for all $i\in\cN\setminus\cK$, then due to Part (2), for the subsystem consisting of finite queues, the (scaled) hitting time to the (unscaled) state $z$ converges in probability to 0.
Also, since this subsystem is positive recurrent (due to Part (2)), starting from a fixed (unscaled) state~$z$, its expected (unscaled) return time to the state~$z$ is $O(1)$.
This will allow us to split the (unscaled) time line into i.i.d.~renewal cycles of finite expected lengths.
In addition, this also shows that in the scaled time the subsystem of finite queues evolves on a faster time scale and achieves `instantaneous stationarity'.
\item From the above observation we can claim that the number of arrivals to any specific infinite queue can be written as a sum of arrivals in the above defined i.i.d.~renewal cycles.
Using the strong law of large numbers (SLLN) we can then show that in the limit $R\to \infty$, the instantaneous rate of arrival to an specific infinite queue is given by the \emph{average arrival rate} when the subsystem with $N-k$ finite queues is in steady state.
Therefore, Part (3) completes the verification of Part (4).
\end{enumerate}
The above argument is rigorously carried out in Subsection~\ref{ssec:part4}.

\subsection{Coupling with infinite queues to verify Part (1)}\label{ssec:part1}
To prove Part (1), fix any $\xx(0)$ such that $0 \le \sum_{i\in\cN\setminus\cK} x_i(0) < \infty$ and $x_i(0) = \infty$ for $i\in\cK$.
Let $\cK_1\subseteq \cN\setminus\cK$ be the set of server-indices~$i$, such that $x_i(0) > 0$.
We will first show that when $\sum_{i\in\cN\setminus\cK} x_i(0) > 0$ with $|\cM(t)|=m\geq k(N)+1$, then it has a negative drift $-\varepsilon(m)$ for all $i\in\cM(t)$, thus proving Part (1.i). 
Since $\varepsilon(m)$'s are positive, this will then also imply Part (1.ii).
Now assume $\sum_{i\in \cN\setminus\cK}x_i(0)>0$. 
In that case we have that $|\cK_1|=:k_1>0$.
Now consider the sequence of processes $(x_i^R(\cdot), a_i^R(\cdot), d_i^R(\cdot))_{i\in\cN}$ along any subsequence $\{R\}$.
Define the stopping time
$$T^R:=\inf\big\{t: X_i^R(t) = 0 \mbox{ for some }i\in\cK_1\big\},$$
and $\tau^R= T^R/R$.
In the time interval $[0,T^R]$, we will couple this system with a system, let us label it $\Pi$, with $k+k_1$ infinite queues.
Let $(\bx_i^R(\cdot), \ba_i^R(\cdot), \bd_i^R(\cdot))_{i\in\cN}$ be the queue length, arrival, and departure processes corresponding to the system $\Pi$, and assume that $\bx_i^R(0)$ is infinite for $i\in \cK\cup \cK_1$. 
Now couple each arrival to and departure from $i$-th server in both systems, $i\in\cN$.
Since the scheme does not distinguish among servers with positive queue lengths, observe that up to time $T^R$ both systems evolve according to their own statistical laws.
Also, up to time $T^R$, the queue length processes at the servers in $\cN\setminus (\cK\cup\cK_1)$ in both systems are identical.
Thus, in the (scaled) time interval $[0,\tau^R]$, $a_i^R \equiv \ba_i^R$ and $d_i^R\equiv \bd_i^R$ for all $i\in\cN$, and $x_i^R\equiv \bx_i^R$ for all $i\in \cN\setminus \cK$.
Therefore, using induction hypothesis for systems with $k+k_1\geq k(N)+1$ infinite queues, there exists a subsequence $\{R\}$ along which with probability 1,
$$(\bx_i^R(\cdot), \ba_i^R(\cdot), \bd_i^R(\cdot))_{i\in\cN}\to (\bx_i(\cdot), \ba_i(\cdot), \bd_i(\cdot))_{i\in\cN},$$
where $\bx_i\equiv 0$ for all $i\in \cN\setminus (\cK\cup\cK_1)$, and $\bx_j\equiv \infty$ with $\bx_j'\equiv -\varepsilon(k+k_1)<0$ for all $j\in\cK\cup\cK_1$.
Consequently, in the time interval $[0,\tau]$, along that subsequence with probability 1,
$$(x_i^R(\cdot), a_i^R(\cdot), d_i^R(\cdot))_{i\in\cN}\to (x_i(\cdot), \ba_i(\cdot), \bd_i(\cdot))_{i\in\cN}$$
with $x_i= \bx_i\equiv 0$ for all $i\in\cN\setminus (\cK\cup\cK_1)$ and $x_i'\equiv -\varepsilon(k+k_1)<0$ for all $i\in\cK\cup\cK_1$,
where $\tau = x_{(k+k_1)}/\varepsilon(k+k_1)>0$.
Observe that the above argument can be extended till the time $\sum_{i\in \cN\setminus\cK}x_i(t)$ hits zero.
Furthermore, following the argument as above, this time is given by $T(k(N),\xx(0))$ as given in~\eqref{eq:Tkx}.
This completes the proof of Part 1 (iii).

\subsection{Conventional fluid-limit stability to verify Part (2)}\label{ssec:part2}
As mentioned earlier, we will use the fluid limit technique of proving stochastic stability as in~\cite{RS92, S95, D99} to prove Part (2).
Consider a sequence of initial states with increasing norm $R$, i.e., $\sum_{i\in\cN\setminus\cK}X_i^{R}(0) =R$ and $X_i^R(0)=\infty$ for $i\in\cK$.
Then from Part (1.iii), we know that for any sequence there exists a further subsequence $\{R\}$ along which with probability 1, the fluid-scaled occupancy process $(x_i^R(\cdot))_{i\in\cN}$ converges to the process $(x_i(\cdot))_{i\in\cN}$ 
for which $\sum_{i\in\cN\setminus\cK}x_i(t)$ hits 0 in finite time $T(k(N), \xx(0))$, and stays at 0 afterwards.
This verifies the fluid-limit stability condition in~\cite[Theorem 4.2]{D99} and \cite[Theorem 7.2]{S95}, and thus completes the verification of Part (2).

\subsection{Large-scale asymptotics to verify Part (3)}\label{ssec:part3}
The verification of the backward induction step for Part~(3) uses contradiction.
Namely, assuming that the induction step for Part~(3) does not hold, we will construct a sequence of systems with increasing~$N$, for which we obtain a contradiction using Lemmas~\ref{lem:expo-bound} and~\ref{lem:steady-concentration}.
We note that this is the only part in the proof of Proposition~\ref{prop:largeN}, where we use the large-scale (i.e., $N\to\infty$) asymptotic results.

Observe that we have already argued in Subsection~\ref{ssec:roadmap} that \emph{for all} $N$, Parts (1) -- (4) hold for $k=N$.
Now, if for some $N$, Part (3) does not hold for some $k(N)\in \{0,1,\ldots, N-1\}$ while Parts (1)--(4) hold for all $k\geq k(N)+1$, then from the proofs of Parts (1) and (2), note that Parts (1) and (2) hold for $k=k(N)$ as well.
Consequently, the subsystem with $N-k(N)$ finite queues is stable.
Thus we have the following implication.

\begin{implication}\label{impl:part3not}
Suppose, for infinitely many $N$, the induction step to prove Part~(3) of Proposition~\ref{prop:largeN} does not hold for some $k=k(N)$.
 Then there exists a subsequence of $\{N\}$ $($which we still denote by $\{N\})$ diverging to infinity, such that (i) The system with $k(N)$ infinite queues is stable and (ii) The steady-state arrival rate into each infinite queue is at least 1.
\end{implication}

We will now show that Implication~\ref{impl:part3not} leads to a contradiction -- this will prove Part~(3) of Proposition~\ref{prop:largeN}.
Suppose Implication~\ref{impl:part3not} is true. 
Choose a further subsequence $\{N\}$ along which $k(N)/N$ converges to $\kappa\in [0,1]$.
As in the statement of Lemma~\ref{lem:steady-concentration} we will consider two regimes depending on whether $\kappa\geq 1-\lambda$ or not, and arrive at contradictions in both cases.
Since all the infinite queues are exchangeable, we will use $\sigma$ to denote a typical infinite queue.\\

\noindent
\textbf{Case 1.} 
First consider the case when $\kappa\geq 1-\lambda$.
Note that the expected steady-state instantaneous rate of arrival to $\sigma$ is given by
\begin{equation}\label{eq:inst-rate}
\begin{split}
&\expt\Big(\frac{\lambda N}{Q_1^{N}(\infty)}\ind{Q_1^N(\infty)+\Delta_0^N(\infty)+\Delta_1^N(\infty)=N}\Big)\leq \expt\Big(\frac{\lambda N}{Q_1^{N}(\infty)}\Big)= \expt\Big(\frac{\lambda}{q_1^{N}(\infty)}\Big)+o(1).
\end{split}
\end{equation}
Now observe that for large $N$, $\lambda/q_1^N(\infty)\leq 2\lambda/\kappa$, since $q_1^N(s)\geq \kappa/2>0$.
Further from Lemma~\ref{lem:steady-concentration} we know that $q_1^N(\infty)\pto 1$ as $N\to \infty$.
Consequently, $\expt(\lambda/q_1^N(\infty))\to \lambda$ as $N\to\infty.$
Therefore for large enough $N$,
\begin{equation}\label{eq:arrival1}
\begin{split}
&\expt\Big(\frac{\lambda N}{Q_1^{N}(\infty)}\ind{Q_1^N(\infty)+\Delta_0^N(\infty)+\Delta_1^N(\infty)=N}\Big)
\leq \frac{1+\lambda}{2} = 1- \frac{1-\lambda}{2}<1,
\end{split}
\end{equation}
which is a contradiction to Part (ii) of Implication~\ref{impl:part3not}.\\

\noindent
\textbf{Case 2.} 
In case $\kappa<1-\lambda$, first note that the statement in Part (3) is vacuously satisfied if $k(N) \equiv 0$ for all large enough $N$.
Thus without loss of generality, assume $k(N)>0$.
Fix $\varepsilon_1$ as in Lemma~\ref{lem:expo-bound}.
In that case \eqref{eq:inst-rate} becomes
\begin{align*}
&\expt\Big(\frac{\lambda N}{Q_1^{N}(\infty)}\ind{Q_1^N(\infty)+\Delta_0^N(\infty)+\Delta_1^N(\infty)=N}\Big) \\
&\leq \expt\Big(\frac{\lambda N}{Q_1^{N}(\infty)}\ind{Q_1^N(\infty)+\Delta_0^N(\infty)+\Delta_1^N(\infty)=N,\ Q_1^N(\infty)\geq \varepsilon_1 N}\Big) 
+ \lambda N \Pro{Q_1^{N}(\infty)<\varepsilon_1 N}\\
&\leq \frac{\lambda N}{\varepsilon_1 N}\Pro{Q_1^N(\infty)+\Delta_0^N(\infty)+\Delta_1^N(\infty)=N} 
+ \lambda N \Pro{Q_1^{N}(\infty)<\varepsilon_1 N}.
\end{align*}
Now, due to Part~(2) of Lemma~\ref{lem:steady-concentration}, we know that 
\[\Pro{Q_1^N(\infty)+\Delta_0^N(\infty)+\Delta_1^N(\infty)=N}\to 0,\] 
and furthermore, Lemma~\ref{lem:expo-bound} yields 
\[N \Pro{Q_1^{N}(\infty)<\varepsilon_1 N}\to 0\quad \mbox{as}\quad  N\to\infty.\]
Thus, 
\begin{align}\label{eq:arrival2}
&\expt\Big(\frac{\lambda N}{Q_1^{N}(\infty)}\ind{Q_1^N(\infty)+\Delta_0^N(\infty)+\Delta_1^N(\infty)=N}\Big) \to 0 \quad\mbox{as}\quad N\to\infty.
\end{align}
In particular, for large enough $N$, the expected steady-state arrival rate is bounded away from 1, which is again a contradiction to Part (ii) of Implication~\ref{impl:part3not}. 
This completes the verification of Part (3) of the backward induction hypothesis.

\subsection{Time-scale separation to verify Part (4)}\label{ssec:part4}
Assume Parts (1) -- (3) hold for all $k\in \{k(N), k(N)+1,\ldots, N\}$.
Now consider a system containing $k=k(N)$ infinite queues with indices in $\cK$, and recall the conventional fluid scaling and FSP from Subsection~\ref{ssec:fluid}.
Also, in this subsection whenever we refer to the process $\{\XX(t)\}_{t\geq 0}$, the components in $\cK$ should be taken to be infinite.

For the queue length vector $\XX$, define the norm $\|\XX\|:=\sum_{i\notin\cK}X_i$ to be the total number of tasks at the finite queues.
Lemmas~\ref{lem:norm-decay} and~\ref{lem:hitting} state two hitting time results that will be used in verifying Part (4).
\begin{lemma}\label{lem:norm-decay}
For any fixed $\gamma\in (0,1)$, there exists $\tau=\tau(\gamma)$ and $C=C(\gamma)$, such that 
if $\|\XX(0)\|=R\geq C,$ then \[\expt\|X(R\tau)\|\leq (1-\gamma)\|X(0)\|.\]
\end{lemma}
Lemma~\ref{lem:norm-decay} says that if the system starts from an initial state where the total number of tasks in the finite queues is suitably large, then the time it takes when the expected total number of tasks in the finite queues falls below a certain fraction of the initial number, is proportional to itself.
The proof of Lemma~\ref{lem:norm-decay} is fairly straightforward, but is provided below for completeness.
\begin{proof}[Proof of Lemma~\ref{lem:norm-decay}]
Consider a sequence of initial states with an increasing norm, i.e., $\XX^R(0)$ is such that $\|\XX^R(0)\|=R$ where $R^{-1}\XX^R(0)\to\xx(0)$ as $R\to\infty$.
Then from Part 1 we know that as $R\to\infty$, on the time interval $[0,T(m,\xx(0))]$ the process $R^{-1}\XX^R(Rt)$ converges in probability to the unique deterministic process $\xx(t)$ satisfying
\begin{equation}\label{eq:drift}
\sum_{i\in\cN\setminus\cK}x_i'(t)<-c(k(N)) 
\quad \mbox{whenever}\quad \sum_{i\in\cN\setminus\cK}x_i(t)>0,
\end{equation}
where $c(m)= \min\big\{k\varepsilon(k):k(N)+1\leq k\leq N\big\}>0$.
We also know that for any $i\in\cN\setminus\cK$, if $x_i(t_0)=0$ for some $t_0$, then $x_i(t)=0$ for all $t\geq t_0$.
Consequently, since $c(k(N))$ is positive, there exists $\tau= \tau(\gamma)<\infty$, such that 
\[\sup_{\|\xx\|=1}\Big\{\|\xx(\tau)\|: \xx(0) = \xx\in[0,1]^{N-k(N)}\times\{\infty\}^{k(N)}\Big\}<1-\gamma.\]
Now since the expected number of arrivals into the $R$-th system up to time $t$, when scaled by $R$, is $\lambda t$ for any finite $t$, we obtain $\expt(R^{-1}\|X^R(t)\|)\leq 1+\lambda t.$
Therefore, the convergence in probability also implies the convergence in expectation.
Thus for the above choice of $\gamma$,
\[\limsup_{R\to\infty}\expt\Big(\frac{\|\XX^R(R\tau)\|}{R}\Big)<1-\gamma.\]
Hence, there exists $C$ such that for all $R\geq C$,
\[\expt\Big(\frac{\|\XX^R(R\tau)\|}{R}\Big)=\expt\Big(\frac{\|\XX^R(R\tau)\|}{\|\XX^R(0)\|}\Big)\le 1-\gamma.\]
This completes the proof of Lemma~\ref{lem:norm-decay}.
\end{proof}

For any $C>0$, define the set $\cC:=\{\|\XX\|\leq C\}$, and the stopping time $\theta_C := \inf\ \{t: \XX(t)\in\cC\}$.
For large enough $C$, the next lemma bounds the expected hitting time to the fixed set $\cC$ in terms of the norm of the initial state.
\begin{lemma}\label{lem:hitting}
There exists $C, C_1>0$, such that if $\|\XX(0)\|=R\geq C,$ then \[\expt(\theta_C|\XX(0))\leq C_1\|\XX(0)\|.\]
\end{lemma}
\begin{proof}[Proof of Lemma~\ref{lem:hitting}]
Fix any $\gamma\in (0,1)$, and take $\tau = \tau(\gamma)$ and $C=C(\gamma)$ as in Lemma~\ref{lem:norm-decay}.
For $i\geq 1$, define the sequence of random variables $T_i:=\tau \|\XX(T_{i-1})\|$ with the convention that $T_0=0$.
Now consider the discrete time Markov chain $\{\Phi_i:i\geq 0\}$ adapted to the filtration $\boldsymbol{\cF}=\bigcup_{i\geq 0}\cF_i$, where $\Phi_i= \XX(T_i)$ is the value of the continuous time Markov process sample at times $T_i$'s, and $\cF_i=\sigma(\Phi_0,\Phi_1,\ldots,\Phi_i)$ is the sigma field generated by $\{\Phi_0,\Phi_1,\ldots,\Phi_i\}$.
Further, for $i\geq 0$ define the stopping time
$\htheta_C:= \inf\ \{j\geq 0:Z_j\leq C\}$.
Then observe that 
\[\theta_C\leq \sum_{i=1}^{\htheta_C}T_i=:\Psi_C.\]
Also define $\alpha_i = \sum_{j=1}^iT_j$ for $i\geq 1$, and hence $\alpha_{\htheta_C}=\Psi_C.$
Then as a consequence of Dynkin's lemma~\cite[Theorem 11.3.1]{MT93}, using~\cite[Proposition 11.3.2]{MT93} we have  
\[\expt(\theta_C)\leq \expt(\Psi_C)\leq \frac{\tau}{\gamma} \expt\|\XX(0)\|.\]
Choosing $C_1 = \tau/\gamma$ completes the proof.
\end{proof}
Now we have all the ingredients to verify Part (4) of the backward induction hypothesis.
Note that we now look at the sequence of conventional fluid-scaled processes starting at (scaled) time $T(k(N),\xx(0))$.
From the verification of Part (1) we already know that $x_i(t)= 0$ for all $t\geq T(k(N),\xx(0))$, $i\in\cN\setminus\cK$.
Thus, it only remains to show that starting from time $T(k(N),\xx(0))$, the drift of each of the infinite queues is at most $-\varepsilon(k(N))$.
Specifically, we will construct a probability space where the required probability 1 convergence holds.

In order to simplify writing, we assume that the system starts at time 0, and thus it is enough to consider a sequence of initial queue length vectors such that
\[\|\xx^R(0)\|\to 0\quad \mbox{as}\quad R\to\infty,\]
where $R$ is the parameter in the conventional fluid scaling.
Hence, Lemma~\ref{lem:hitting} yields that $R^{-1}\expt(\theta_C|\XX^R(0))\to 0$ as $R\to\infty$.
Consequently, $R^{-1}\theta_C\pto 0.$
Thus, the fluid-scaled time to hit the set~$\cC$ vanishes in probability, which is stated formally in the following claim.
\begin{claim}\label{claim:hittingC}
If the sequence of initial states is such that $\|\xx^R(0)\|\to 0$ as $R\to\infty$, then
$R^{-1}\theta_C\pto 0,$ as $R\to\infty$.
\end{claim}
\noindent
Now pick any (unscaled) state $z\in\cC$, and define the stopping time $\htheta_z$ as 
\[\htheta_z:=\inf\big\{t\geq 0: \XX(t) = z\big\}.\]
Since due to Part (2) of the backward induction hypothesis, the unscaled process $\XX(\cdot)$ is irreducible and positive recurrent, we have the following claim.
\begin{claim}\label{claim:hittingz}
If the sequence of initial states is such that $\xx^R(0)\in\cC$, then $R^{-1}\htheta_z\pto 0$, as $R\to\infty.$
\end{claim}
\noindent
Up to time $\htheta_z$, consider the product topology on the sequence space.
Then Claims~\ref{claim:hittingC} and~\ref{claim:hittingz} yield that for a sequence of initial states such that $\|\xx^R(0)\|\to 0$ as $R\to\infty$, there exists a subsequence $\{R\}$, along which with probability 1, $R^{-1}\htheta_z\to 0$.
Starting from the time $\htheta_z$, along the above subsequence, we construct the sequence of processes $\xx^R(\cdot)$ on the same probability space as follows.\\

\noindent
(1) Define the space of an infinite sequence of i.i.d.~renewal cycles of the unscaled process $\XX(\cdot)$, with the unscaled state $z$ being the renewal state, i.e.,
\[\Big\{\XX^{(i)}(t): 0\leq t\leq \htheta_z^{(i)}, \XX^{(i)}(0)=z\Big\}\] 
for $i=1,2,\ldots$ are i.i.d.~copies, and $\htheta_z^{(i)}$ are also i.i.d.~copies of $\htheta_z$.\\

\noindent
(2) Define the process $\XX^R(\cdot)$ as 
\[\XX^R(Rt)= \sum_{i=1}^\infty \XX^{(i)}\big(Rt- \Theta(i-1)\big)\ind{\Theta(i-1)\leq Rt<\Theta(i)}, \quad \text{where}\quad
\Theta(i):= \sum_{j=1}^i\htheta^{(j)}.\]

\noindent
Let $A(t)$ denote the cumulative number of arrivals up to time $t$ to a fixed server with infinite queue length when the system starts from the state $z$.
Now, in order to calculate the drift of each of the infinite queues, 
observe that cumulative number of arrivals up to time $Rt$ to server  $n\in\cK$ in the $R$-th system can be written as
\begin{align*}
A_n^R(Rt) = \sum_{i=1}^{N_\theta^R} A_n^{(i)}+B_n(t-\Theta(N_\theta^R)), \quad \text{where}\quad N_\theta^R:= \max\{j:\Theta(j)\leq Rt\}.
\end{align*}
 $A_n^{(i)}$'s are i.i.d.~copies of the random variable $A(\htheta_z)$, $B_n(\cdot)$ is distributed as $A(t)$, and $A_n^{(i)}$'s and $B_n(\cdot)$ are independent of the random variable $N_\theta^R$.
 Now, since due to Part (2) of the backward induction hypothesis the subsystem consisting of the finite queues is stable, $\XX(\cdot)$ is irreducible and positive recurrent.
 Thus, we have $\expt(\htheta_z|\XX(0)=z)<\infty$, and hence, with probability 1, 
 \[\frac{N_\theta^R}{R}\to \frac{t}{\expt(\htheta_z|\XX(0)=z)}, \quad\mbox{as}\quad R\to\infty.\]
 Thus, using Part (3) of the backward induction hypothesis, SLLN yields, with probability~1,
 \begin{align*}
 \frac{1}{R}A_n^R(Rt) &= \frac{1}{R}\sum_{i=1}^{N_\theta^R} A_n^{(i)}+\frac{B_n(t-\Theta(N_\theta^R))}{R}\to \ha t, \quad\mbox{as}\quad R\to\infty,
 \end{align*}
 for some $\ha\leq 1-\varepsilon(k(N))$.
 Therefore, in the conventional fluid limit, $a_n(t)\leq  (1-\varepsilon(k(N)))t$. 
 Also, since the departure rate from each of the servers with infinite queue lengths is always~1, it can be seen that in the conventional fluid limit, $d_n(t) = t$, and thus, the drift of the $n$-th infinite queue is given by at most $-\varepsilon(k(N))$.
Combining the probability 1 convergence of the time $\htheta_z$ to 0, and the probability space constructed after time $\htheta_z$, we obtain that along the subsequence $\{R\}$ with probability 1, the fluid-scaled processes converges to a limit where each infinite queue has drift at most $-\varepsilon(k(N))$.
This completes the verification of Part (4), and hence of Proposition~\ref{prop:largeN}.

\section{Mean-field analysis for large-scale asymptotics}\label{sec:mf}
In this section we will analyze the large-$N$ behavior of the system.
In particular, we will prove Lemmas~\ref{lem:expo-bound} and~\ref{lem:steady-concentration}.
The next proposition is a basic mean-field fluid limit result that we need later.
Define 
\[E_\kappa := \Big\{(\bld{q},\dd)\in [0,1]^\infty:  q_i\geq q_{i+1}\geq \kappa,\ \forall i, \  \delta_0+\delta_1+ q_1\leq 1 \Big\}.\]

\begin{proposition}
\label{prop:mf}
Assume $k(N)/N\to\kappa\in [0,1]$ and the sequence of initial states $(\qq^N(0),\bld{\delta}^N(0))$ converge to a fixed $(\qq(0),\bld{\delta}(0))\in E_\kappa$, as $N\to\infty$,  where $q_1(0)>0$. 
Then, 
with probability 1, any subsequence of $\{N\}$ has a further subsequence along which $\{(\qq^N(t),\dd^N(t))\}_{t\geq 0}$ converges, uniformly on compact time intervals, to some deterministic trajectory $\{(\qq(t),\dd(t))\}_{t\geq 0}$ satisfying the following equations:
\begin{align*}
 q_i(t)&=q_i(0) +\int_0^t\lambda p_{i-1}(\qq(s),\dd(s),\lambda)\dif s
- \int_0^t(q_i(s)-q_{i+1}(s))\dif s,\ i\ge 1 \\
\delta_0(t)&=\delta_0(0)+\mu\int_0^t u(s)\dif s-\xi(t),  \\
\delta_1(t)&=\delta_1(0)+\xi(t)-\nu\int_0^t\delta_1(s)\dif s,\nonumber
\end{align*}
where 
\begin{align*}
u(t) &= 1- q_1(t) - \delta_0(t) - \delta_1(t),\\
\xi(t) &= \int_0^t\lambda(1-p_0(\qq(s),\dd(s),\lambda))\ind{\delta_0(s)>0}\dif s.
\end{align*}
For any $(\qq,\dd)\in E$, $\lambda>0$, $(p_i(\qq,\dd,\lambda))_{i\geq 0}$ are given by 
\begin{align*}
 p_0(\qq,\dd,\lambda) &= 
\begin{cases}
&1\qquad \text{if}\qquad u=1-q_1-\delta_0-\delta_1>0,\\
&\min \{\lambda^{-1}(\delta_1\nu + q_1-q_2), 1\},\quad\text{otherwise,}
\end{cases}\\
\quad p_i(\qq,\dd,\lambda)  &= (1-p_0(\qq,\dd,\lambda)) (q_{i}-q_{i+1})q_1^{-1},\  i \ge 1.
\end{align*}
\end{proposition}
This type of result is standard and is obtained using Functional Strong LLN, for example as in~\cite{Stolyar15,Stolyar17, MDBL17}; we omit its proof.
Also, we note that, while Proposition~\ref{prop:mf} is a version of \cite[Theorem 3.1]{MDBL17}, it is different in that it is suitably modified for the case of infinite buffers and some queues being infinite, and it states a somewhat different type of convergence, convenient for the use in this paper.
Define \emph{mean-field fluid sample path} (MFFSP) to be any deterministic trajectory satisfying the properties stated in Proposition~\ref{prop:mf}.

We now provide an intuitive explanation of the mean-field fluid limit stated in Proposition~\ref{prop:mf}.
It is similar to that behind~\cite[Theorem 3.1]{MDBL17}.
The term $u(t)$ corresponds to the asymptotic fraction of idle-on servers in the system at time $t$, and $\xi(t)$ represents the asymptotic cumulative number of server setups (scaled by $N$) that have been initiated during $[0,t]$.
The coefficient $p_i(\qq,\dd,\lambda)$ can be interpreted as the instantaneous fraction of incoming tasks that are assigned to some server with queue length $i$, when the fluid-scaled occupancy state is $(\qq,\dd)$ and the scaled instantaneous arrival rate is $\lambda$.
Observe that as long as $u>0$, there are idle-on servers, and hence all the arriving tasks
will join idle servers. 
This explains that if $u>0$, $p_0(\qq,\dd,\lambda) = 1$ and $p_i(\qq,\dd,\lambda)=0$ for $i=1,2,\ldots$.
If $u=0$, then observe that 
servers become idle at rate $q_1-q_2$, and servers in setup mode turn on at rate $\delta_1\nu$.
Thus the  idle-on servers are created at a total rate $\delta_1\nu + q_1-q_2$.
If this rate is larger than the arrival rate $\lambda$, then almost all the arriving tasks can be assigned to idle servers.
Otherwise, only a fraction $(\delta_1\nu + q_1-q_2)/\lambda$
of arriving tasks join idle servers. 
The rest of the tasks are distributed uniformly among busy servers, so a proportion $(q_{i}-q_{i+1})q_1^{-1}$ are assigned to servers having queue length~$i$.
For any $i=1,2, \ldots$, $q_i$ increases when there is an arrival to some server with queue length $i-1$, which occurs at rate $\lambda p_{i-1}(\qq,\dd,\lambda)$, and it decreases when there is a departure from some server with  queue length~$i$, which occurs at rate $q_i-q_{i-1}$. 
Since each idle-on server turns off at rate $\mu$, the fraction of servers in the off mode increases at rate 
$\mu u$.
Observe that if $\delta_0>0$, for each task that cannot be assigned to an idle server, a setup procedure is initiated  at one idle-off server. 
As noted above, $\xi(t)$ captures the (scaled) cumulative number of setup procedures initiated up to time~$t$.
Therefore the fraction of idle-off servers and the fraction of servers in setup mode decreases and increases by $\xi(t)$, respectively, during $[0,t]$.
Finally, since each server in setup mode becomes idle-on at rate $\nu$, the fraction of servers in setup mode decreases at rate $\nu\delta_1$.

\subsection{Proof of \texorpdfstring{Lemma 3.2}{Lemma~\ref{lem:expo-bound}}}
Throughout this subsection we will prove Lemma~\ref{lem:expo-bound}.
Within this proof we will use the following terminology. 
Let $A^N$ be an event pertaining to $N$-th system. We will write $\Pro{A^N} = \eta(N)$ to mean the following property: \emph{There exist $C>0$  and $N_1>0$ such that $\Pro{A^N} \le \e^{-C N}$ for all $N \ge N_1$.}
If event $A^N$ depends on some parameter $p$ (say, the process initial state), we say that $\Pro{A^N} = \eta(N)$ \emph{uniformly in $p$} if the 
property holds for common fixed $C>0$  and $N_1>0$.

To prove the lemma, clearly, it suffices to prove that for some fixed $T_0>0$ and $\varepsilon_0>0$
\begin{equation}\label{eq:q-1-lower}
\Pro{q_1^N(T_0)\leq \varepsilon_0} = \eta(N),
\end{equation}
uniformly on the process initial states $(\qq^N(0),\dd^N(0))$. This is what we do in the rest of the proof.

Fix any $T_0>0$; $\varepsilon_0>0$ will be chosen later. We now prove several claims, which rather simply follow from 
the process structure and
basic large deviations estimates (specifically, Cramer's theorem) -- they will serve as building blocks for the proof argument.

\begin{claim}
\label{claim-busy-decay}
{\normalfont(i)} For any $\varepsilon>0$, uniformly in $\tau\in [0,T_0]$ and uniformly in $q_1^N(0) \ge \varepsilon$, 
\begin{equation}
\label{eq-busy-decay}
\Pro{q_1^N(\tau) \le (\varepsilon/2) \e^{-T_0}} = \eta(N).
\end{equation}
{\normalfont(ii)} For any $\varepsilon>0$, uniformly in $\tau\in [0,T_0]$ and uniformly in $\delta_1^N(0) \ge \varepsilon$, 
\begin{equation}
\label{eq-setup-decay}
\Pro{\delta_1^N(\tau) \le (\varepsilon/2) \e^{-\nu T_0}} = \eta(N).
\end{equation}
\end{claim}

\noindent Indeed, to prove (\ref{eq-busy-decay}), observe that
any busy server at time $t$ stays busy in the interval $[t,t+\tau]$ with probability at least $\e^{-\tau}\ge \e^{-T_0}$. It remains to recall that
$q_1^N(0) \ge \varepsilon$ corresponds to at least $\varepsilon N$ busy servers in the unscaled system and apply Cramer's theorem. 
Statement~(ii) is proved analogously.

\begin{claim}
\label{claim-T1}
For any sufficiently small $T_1>0$, there exists $\varepsilon'_1>0$ such that, uniformly in the initial state $(\qq^N(0),\dd^N(0))$,
\begin{equation}
\label{eq-T1}
\Pro{q_1^N(T_1) + \delta_1^N(T_1) \le \varepsilon'_1} = \eta(N).
\end{equation}
\end{claim}

\noindent Indeed, fix any $T_1>0$ such that $\lambda T_1 \le 1/4$. Suppose first that either $q_1^N(0)\ge 1/4$ or $\delta_1^N(0)\ge 1/4$; uniformly on all such initial conditions, the claim follows by using Claim~\ref{claim-busy-decay}. 
Suppose now that $q_1^N(0)< 1/4$ and $\delta_1^N(0)< 1/4$, and therefore $\delta_0^N(0)+u^N(0) > 1/2$, where recall that $u^N$ is the fraction of idle-on servers. 
The (unscaled) number of new customer arrivals in $[0,T_1]$, denote it by $H[0,T_1]$, is Poisson with mean $\lambda T_1 N$; therefore, 
$$\Pro{| H[0,T_1]/N - \lambda T_1 | \ge (1/2) \lambda T_1} = \eta(N).$$ 
This means that with probability $1-\eta(N)$,  
we have $H[0,T_1]/N < \delta_0^N(0)+u^N(0)$, and therefore
each arrival in $[0,T_1]$ creates either a new busy server or a new setup server; furthermore, each of these newly created busy or setup servers will not change its state until time $T_1$ with probability at least $\e^{-\nu' T_1}$, where $\nu' = \max \{\nu,1\}$.
It remains to choose $\varepsilon'_1 \in (0, (1/4) \lambda T_1 \e^{-\nu' T_1})$ to obtain the claim.

\begin{claim}
\label{claim-T2}
For any $\varepsilon_1>0$ and any $T_2>0$, there exists $\varepsilon'_2>0$ such that, uniformly in $\delta_1^N(0) \ge \varepsilon_1$,
\begin{equation}
\label{eq-T2}
\Pro{q_1^N(T_2) + u^N(T_2) \le \varepsilon'_2} = \eta(N).
\end{equation}
\end{claim}

\noindent Indeed, at time $0$ there are at least $\varepsilon_1 N$ setup servers. Fix any $T_2>0$. In $[0,T_2]$ each of them tuns into an idle-on server with probability at least $1-\e^{-\nu T_2}$; those servers that do turn into idle-on will be either still be idle-on or busy at time $T_2$ with probability at least $\e^{-\nu'' T_2}$, where $\nu'' = \max \{\mu,\nu\}$.
It remains to choose $\varepsilon'_2 \in (0, (1/2) \varepsilon_1 \e^{-\nu'' T_2})$, and apply Cramer's theorem.

\begin{claim}
\label{claim-T3}
For any $\varepsilon_2>0$ and any sufficiently small $T_3>0$, there exists $\varepsilon_3>0$ such that, uniformly in $u^N(0) \ge \varepsilon_2$,
\begin{equation}
\label{eq-T3}
\Pro{q_1^N(T_3) \le \varepsilon_3} = \eta(N).
\end{equation}
\end{claim}

\noindent Indeed, fix $T_3$ small enough so that $\e^{-\mu T_3}>3/4$ and $\lambda T_3 < \varepsilon_2/2$. 
At time $0$ there are at least $\varepsilon_2 N$ idle-on servers; with probability at least $\e^{-\mu T_3}>3/4$ they will still be
idle-on at time $T_3$, \emph{unless} they are taken by a new arrival. The (unscaled) number of new arrivals in $[0,T_3]$, namely
$H[0,T_3]$, is Poisson with mean $\lambda T_3 N$, and therefore $H[0,T_3]/N$ concentrates at $\lambda T_3$:
$\Pro{| H[0,T_3]/N - \lambda T_3 | \ge (1/2) \lambda T_3} = \eta(N)$. We conclude that with probability $1-\eta(n)$ every new arrival
in $[0,T_3]$ will go to an idle-on server and turn it into busy; each of those servers, in turn, will remain busy until $T_3$ with 
probability at least $\e^{-T_3}$. 
It remains to choose $\varepsilon_3 \in (0, (1/4) \lambda T_3 \e^{-T_3})$ to obtain the claim.

With these claims, we are now in position to conclude the proof of the lemma. Choose small $T_1>0$ and $\varepsilon'_1>0$
as in Claim~\ref{claim-T1}; and then $\varepsilon_1=\varepsilon'_1/2$. 
For the chosen $\varepsilon_1$,
choose small $T_2>0$ and $\varepsilon'_2>0$
as in Claim~\ref{claim-T2}; and then $\varepsilon_2=\varepsilon'_2/2$. Finally, for the chosen $\varepsilon_2$, choose small $T_3>0$ and $\varepsilon_3>0$
as in Claim~\ref{claim-T3}. Note that $T_1, T_2, T_3$ can small enough so that $T'_3\doteq T_1+T_2+T_3 \le T_0$; 
let us also denote
$T'_2=T_1+T_2$. Choose $\varepsilon_0 = (1/2) \min\{\varepsilon_1,\varepsilon_2,\varepsilon_3\} \e^{-T_0}$.

According to Claim~\ref{claim-T1}, with probability $1-\eta(N)$, at time $T_1$ we have either $q_1^N(T_1) \ge \varepsilon_1$ or
$\delta_1^N(T_1) \ge \varepsilon_1$. Conditioned on a state at $T_1$ safisfying $q_1^N(T_1) \ge \varepsilon_1$, we have 
(\ref{eq:q-1-lower}) by applying Claim~\ref{claim-busy-decay}. Therefore, it remains to prove (\ref{eq:q-1-lower}) conditioned on a state at $T_1$ satisfying $\delta_1^N(T_1) \ge \varepsilon_1$. Under this condition at $T_1$, we obtain from Claim~\ref{claim-T2} that, with probability $1-\eta(N)$, at time $T'_2$ we have either $q_1^N(T'_2) \ge \varepsilon_2$ or
$u^N(T'_2) \ge \varepsilon_2$. Then, conditioned on a state at $T'_2$ satisfying $q_1^N(T'_2) \ge \varepsilon_2$, we have 
(\ref{eq:q-1-lower}) by once again applying Claim~\ref{claim-busy-decay}. It now remains to prove (\ref{eq:q-1-lower}) conditioned on a state at $T'_2$ satisfying $u^N(T'_2) \ge \varepsilon_2$. Under this condition at $T'_2$, we obtain from Claim~\ref{claim-T3} that, with probability $1-\eta(N)$, at time $T'_3$ we have $q_1^N(T'_3) \ge \varepsilon_3$; and conditioned on  $q_1^N(T'_3) \ge \varepsilon_3$
at $T'_3$, we have (\ref{eq:q-1-lower}) by, yet again, Claim~\ref{claim-busy-decay}. The proof is complete.

\subsection{Proof of \texorpdfstring{Lemma 3.3}{Lemma~\ref{lem:steady-concentration}}}
Throughout this subsection we will prove Lemma~\ref{lem:steady-concentration}.
Recall that the stability of the subsystem $\cN\setminus\cK$ is assumed,
and hence there exists a unique stationary distribution for each $N$. 
Recall that we denote by $\qq^N(\infty)$ the random value of $\qq^N(t)$ in the steady-state. 
We will start by stating a few basic facts about the mean-field limits that will facilitate the proof of Lemma~\ref{lem:steady-concentration}.

Recall the definition of MFFSP from the paragraph after Proposition~\ref{prop:mf}, and that $u(t)=1-q_1(t)-\delta_0(t)-\delta_1(t)$. 
Also, denote by $y_1(t) = q_1(t)-q_2(t)$ and by
$(d^+/dt)$ the right derivative.
\begin{claim}\label{fact:mfl}
For any $\varepsilon>0$ there exists $\alpha>0$, such that any MFFSP with $q_1(0)>0$ satisfies the following properties for all $t\geq 0$:
\begin{enumerate}[{\normalfont (i)}]
\item If $y_1(t) \le \lambda-\varepsilon$ and $u(t)>0$, then $(d^+/dt)q_1(t) \ge \alpha$.
\item If $y_1(t) \le \lambda-\varepsilon$, $u(t)=0$ and $\delta_1(t)\ge \varepsilon$, then $(d^+/dt)q_1(t) \ge \alpha$.
\item If $y_1(t) \le \lambda-\varepsilon$, $u(t)=0$, $\delta_1(t)=0$, and $\delta_0(t)>0$, then $(d^+/dt)\delta_1(t)\ge \varepsilon$.
\end{enumerate}
\end{claim}
\noindent
\begin{proof}
Fix any $\varepsilon>0$. First observe that since $q_1(0)>0$ and due to Proposition~\ref{prop:mf}, $q_1(0)$ is nondecreasing whenever $q_1(t)-q_2(t)\leq \lambda$, we have
$q_1(t)\geq \min\{q_1(0),\lambda\}>0$ for all $t\geq 0$.
Thus, Proposition~\ref{prop:mf} can be applied for all $t\geq 0$, throughout the MFFSP.
Choose $\alpha = \min\{\varepsilon\nu,\ \varepsilon\}$.

For (i), note that if  $y_1(t) \leq \lambda - \varepsilon$ and $u(t)>0$, then 
$$(d^+/dt)q_1(t) = \lambda - (q_1(t) - q_2(t))\geq \varepsilon\geq \alpha.$$

For (ii), note that if $y_1(t) \leq \lambda - \varepsilon$, $u(t)>0$, and $\delta_1(t)\geq \varepsilon$, then due to Proposition~\ref{prop:mf},
\begin{align*}
(d^+/dt)q_1(t) &= \min\big\{(\delta_1(t)\nu + q_1(t) - q_2(t)), \lambda\big\} - (q_1(t) - q_2(t))\\
	&= \min\big\{\delta_1(t)\nu,\quad \lambda - (q_1(t) - q_2(t))\big\}
	\geq \min\big\{\varepsilon\nu,\ \ \varepsilon\big\} = \alpha.
\end{align*}

Finally, for (iii), note that from Proposition~\ref{prop:mf} if $y_1(t) \le \lambda-\varepsilon$, $u(t)=0$, $\delta_1(t)=0$, and $\delta_0(t)>0$, then $(d^+/dt)\delta_1(t) =\lambda-(q_1(t) - q_2(t)) \ge \varepsilon$.
\end{proof}

\noindent
\textbf{Proof of statement (1).}  Note that it is enough to prove the following property of any MFFSP:
\begin{claim}\label{claim:q1case1}
Starting from any state $\qq(0)\in E_\kappa$ with $\kappa\geq 1-\lambda$ and $q_1(0)\in [\kappa, 1)$, along any MFFSP we have $$\lim_{t\to\infty}q_1(t)= 1.$$
\end{claim}
\noindent
Indeed, Claim~\ref{claim:q1case1} implies that under the assumption of stability, asymptotically the stationary distribution of $q_1^N(t)$ must concentrate at $q_1^\star=1$, as $N\to\infty$.

\begin{proof}[Proof of Claim~\ref{claim:q1case1}]
We will prove by contradiction.
Note that for the case under consideration, $q_i(t)\geq \kappa$ for all $i\geq 1$ and $t\geq 0$.
Therefore, throughout the proof of Claim~\ref{claim:q1case1} we can assume $q_1(0)\geq \kappa>0$, and can apply Proposition~\ref{prop:mf} and Claim~\ref{fact:mfl}.

Note that if $q_1(t)< 1$, we have $q_1(t) - q_2(t)<1-\kappa\leq\lambda$, and hence due to Claim~\ref{fact:mfl}, $q_1(t)$ is non-decreasing.
Thus if Claim~\ref{claim:q1case1} does not hold, then there exists an $\varepsilon>0$, such that $q_1(t)\leq 1 - \varepsilon\nu$ for all $t\geq 0$, and hence
\begin{equation}\label{eq:localq1bdd}
q_1(t)-q_2(t)\leq \lambda - \varepsilon\nu \quad\mbox{for all}\quad t\geq 0.
\end{equation}
The high-level proof idea is that if $q_1(t)$ remains below 1 by a non-vanishing amount for all $t\geq 0$, then the (scaled) rate $q_1(t) - q_2(t)$ of busy servers turning idle-on would not be high enough to match the (scaled) rate $\lambda$ of incoming jobs.
If there are idle-on servers (as in Claim~\ref{fact:mfl}.(i)) or sufficiently many servers in setup mode (as in Claim~\ref{fact:mfl}.(ii)), then we can still assign incoming tasks to idle-on servers, but this drives up the fraction of busy servers $q_1(t)$ and cannot continue indefinitely since $q_1(t)\leq 1 - \varepsilon\nu$ for all $t\geq 0$.
This means that we cannot initiate an unbounded number of setup procedures (see Equation~\eqref{eq:contr1}).
At the same time, as argued above, we cannot continue assigning tasks to idle-on servers either. Thus, throughout the MFFSP, a positive fraction of the jobs are assigned to busy servers, which initiates an unbounded (scaled) number of setup procedures, and hence the contradiction. \\

\noindent
Define the subset $\cX_\kappa\subseteq E$ as
$$\cX_\kappa:= \Big\{(\qq,\dd)\in E_\kappa: q_1 + \delta_0 + \delta_1 = 1,  \delta_1\nu+q_1-q_2\leq \lambda\Big\},$$
and denote by  $\indn{\cX_\kappa}(\qq(s),\dd(s))$  the indicator of $(\qq(s),\dd(s))\in \cX_\kappa.$
Observe that due to Proposition~\ref{prop:mf}, $q_1(t)$ can be written as
\begin{equation}\label{eq:q1eqn}
\begin{split}
q_1(t)&= q_1(0) +\int_{0}^t\delta_1(s)\nu\indn{\cX_\kappa}(\qq(s),\dd(s))\dif s +\int_{0}^t[\lambda - q_1(s) + q_2(s)] \indn{\cX^c_\kappa}(\qq(s),\dd(s))\dif s.
\end{split}
\end{equation}
Thus,
\begin{align*}
q_1(t)\geq q_1(0) +\int_{0}^t[\lambda - q_1(s) + q_2(s)] \indn{\cX^c_\kappa}(\qq(s),\dd(s))\dif s,
\end{align*}
and~\eqref{eq:localq1bdd} yields there exists positive constant $K_1$, which may depend on $\varepsilon$ such that $\forall\ t\geq 0$
\begin{equation}\label{eq:K1-choice}
\int_0^t \indn{\cX_{\kappa}^c}(\qq(s),\dd(s))\dif s<K_1
\implies \int_0^t\ind{u(s)>0}\dif s<K_1.
\end{equation}
Again from~\eqref{eq:q1eqn} we obtain
\begin{align*}
q_1(t)&\geq  q_1(0) +\int_{0}^t\delta_1(s)\nu\indn{\cX_\kappa}(\qq(s),\dd(s))\dif s\\
&\geq q_1(0)+\nu\int_{0}^t\delta_1(s)\dif s-(\nu+1)\int_{0}^t\indn{\cX_\kappa^c}(\qq(s),\dd(s))\dif s.
\end{align*}
and thus, by \eqref{eq:localq1bdd} and \eqref{eq:K1-choice}, there exist positive constants $K_2$, $K_2'$  which may depend on $\varepsilon$ such that $\forall\ t\geq 0$
\begin{equation}\label{eq:K2-choice}
\begin{split}
 \int_0^t\delta_1(s)\dif s<K_1 &\implies \int_0^t\ind{\delta_1(s)>\frac{\varepsilon}{2}}\dif s<K_2,
  \implies \int_0^t\ind{\delta_1(s)>\frac{\varepsilon\nu}{2}}\dif s<K_2'.
 \end{split}
\end{equation}
Consequently, due to Proposition~\ref{prop:mf}, since $\delta_1(t) = \delta_1(0) + \xi(t) - \nu\int_0^t\delta_1(s)\dif s,$
it must be the case that 
\begin{equation}\label{eq:contr1}
\limsup_{t\to\infty}\xi(t)<\infty.
\end{equation}
Furthermore, since $q_1(t)\leq 1-\varepsilon\nu$ for all $t\geq 0$,
\[\ind{\delta_0(t)=0} \leq \ind{u(t)>0}+\ind{\delta_1(t)\geq \frac{\varepsilon\nu}{2}}.\]
Thus, \eqref{eq:K1-choice} and \eqref{eq:K2-choice} yield $\forall\ t\geq 0$,
\begin{equation}\label{eq:delta0bdd}
\int_0^t\ind{\delta_0(s)=0}\dif s\leq K_1+K_2'.
\end{equation}
Now from Proposition~\ref{prop:mf} observe that
\begin{align*}
\xi(t)&= \int_0^t\lambda(1-p_0(\qq(s),\dd(s),\lambda))\ind{\delta_0(s)>0}\dif s\\
&\geq \int_{0}^t\lambda(1-p_0(\qq(s),\dd(s),\lambda))\ind{\delta_0(s)>0, u(s) = 0,\delta_1(s)\leq \varepsilon/2}\dif s,
\end{align*}
and on the set $\{s:\delta_0(s)>0, u(s) = 0,\delta_1(s)\leq \varepsilon/2\}$ we have  $p_0(\qq(s),\dd(s),\lambda)\leq \lambda^{-1}(\varepsilon\nu/2+q_1(s)-\kappa)$. Therefore,
\begin{equation}\label{eq:xilowerbdd}
\begin{split}
\xi(t)
&\geq \int_{0}^t\Big(\lambda-\frac{\varepsilon\nu}{2}-q_1(s)+\kappa\Big)\ind{\delta_0(s)>0, u(s) = 0,\delta_1(s)\leq \varepsilon/2}\dif s\\
&\geq \int_0^t\Big(\lambda-\frac{\varepsilon\nu}{2}-q_1(s)+\kappa\Big)\dif s - \int_0^t\ind{\delta_0(s)=0}\dif s 
- \int_0^t\ind{u(s)>0}\dif s - \int_0^t\ind{\delta_1(s)>\varepsilon/2}\dif s,
\end{split}
\end{equation}
where the second inequality is due to the fact that $\lambda-\varepsilon\nu/2-q_1(s)\leq \lambda <1$.
Therefore, since $\lambda+\kappa\geq 1$, we have $\lambda-\varepsilon\nu/2-q_1(s)+\kappa\geq \varepsilon\nu/2>0$, and Equations~\eqref{eq:K1-choice}, \eqref{eq:K2-choice}, \eqref{eq:delta0bdd}, and \eqref{eq:xilowerbdd} implies $\liminf_{t\to\infty}\xi(t) = \infty$, which is a contradiction with~\eqref{eq:contr1}.
This completes the proof of Claim~\ref{claim:q1case1}.
\end{proof}

\noindent
\textbf{Proof of statement (2).} First we will establish convergence of $q_1^N(\infty)$ as $N\to\infty$, followed by convergence of $q_2^N(\infty)$, $\delta_0^N(\infty)$, and $\delta_1^N(\infty)$.\\

\noindent
\textbf{Convergence of $q_1^N(\infty)$.} 
First we will show that 
for all $\varepsilon_2>0$, 
\begin{equation}\label{eq:q1limsup}
\limsup_{N\to\infty} P(q_1^N(\infty) < \kappa+\lambda-\varepsilon_2) = 0.
\end{equation}
Due to Lemma~\ref{lem:expo-bound}, note that any limit of stationary distributions is such that with probability~1, $q_1 \ge \varepsilon_1$ for some fixed $\varepsilon_1>0$.
Therefore, throughout the proof of Part (2) of Lemma~\ref{lem:steady-concentration}, it is enough to consider MFFSP so that $q_1(0)\geq \varepsilon_1$, and Proposition~\ref{prop:mf} and Claim~\ref{fact:mfl} can be used.
Thus, for~\eqref{eq:q1limsup}, it is enough to show that any MFFSP has the following property:
\begin{claim}\label{claim:q1lower}
Starting from any state $\qq(0)\in E_\kappa$ with $q_1(0)\in [\varepsilon_1, \kappa+\lambda)$, along any MFFSP we have 
\begin{equation}\label{eq:q1part2}
\liminf_{t\to\infty}q_1(t)\geq \kappa+\lambda.
\end{equation}
\end{claim}
\noindent
Similar arguments as in the proof of Claim~\ref{claim:q1case1} can be used to prove Claim~\ref{claim:q1lower}, for which we omit the details.
Claim~\ref{claim:q1lower} then implies~\eqref{eq:q1limsup}.

Further, observe that since we have assumed that the system is stable, we have 
\begin{equation}\label{eq:exptbdd}
\lim_{N\to\infty}\expt(q_1^N(\infty)) \leq \kappa+\lambda.
\end{equation}
Fix any $\varepsilon_2'>0$. 
Now for all fixed $M>0$, 
\begin{align*}
\expt(q_1^N(\infty))-(\kappa+\lambda) &\geq \varepsilon_2'\prob(q_1^N(\infty) > \kappa+\lambda+\varepsilon_2')-\frac{\varepsilon_2'}{M}\prob\big(\kappa+\lambda-\frac{\varepsilon_2'}{M}\leq q_1^N(\infty) \leq \kappa+\lambda+\varepsilon_2'\big) \\
&\hspace{4.8cm}- \prob\big(q_1^N(\infty) < \kappa+\lambda-\frac{\varepsilon_2'}{M} \big),\\
& \geq \varepsilon_2'\prob(q_1^N(\infty) > \kappa+\lambda+\varepsilon_2')-\frac{\varepsilon_2'}{M}- \prob\big(q_1^N(\infty) < \kappa+\lambda-\frac{\varepsilon_2'}{M} \big),
\end{align*}
and thus, from \eqref{eq:q1limsup} and \eqref{eq:exptbdd} above,
\[\limsup_{N\to\infty} \prob(q_1^N(\infty) > \kappa+\lambda+\varepsilon_2') \leq \frac{1}{M}\quad\mbox{for all }M>0\]
which in conjunction with~\eqref{eq:q1limsup} completes the proof of convergence of $q_1^N(\infty)$.\\

\noindent
\textbf{Convergence of $q_2^N(\infty)$.} 
Note that given the above convergence of $q_1^N(\infty)$ to $\kappa+\lambda$ as $N\to\infty$, the following property of the mean-field limit is sufficient to prove that the sequence of stationary distributions $q_2^N(\infty)$ concentrate at $q_2^\star = \kappa$ as $N\to\infty$:
\begin{claim}\label{claim:q2conv}
For any $\varepsilon_3>0$, there exists a fixed $T_0$ and $\varepsilon_4>0$, such that starting from any state $\qq(0)\in E_\kappa$ with $q_1(0)=\lambda+\kappa$ and $q_2(0)> \kappa+\varepsilon_3$, along any MFFSP we have $q_1(T_0) \geq  \kappa+\lambda + \varepsilon_4$.
\end{claim}
\noindent
Indeed, if the sequence of the stationary distributions were such that 
$$\limsup_{N\to\infty} \prob\big(q_2^N(\infty)>\kappa+\varepsilon_3\big)>0,$$ then Claim~\ref{claim:q2conv} would imply that $\limsup_{N\to\infty} \prob\big(q_1^N(\infty)>\kappa+\lambda+\varepsilon_4/2\big)>0$, which contradicts the convergence of $q_1^N(\infty)$.

\begin{proof}[Proof of Claim~\ref{claim:q2conv}]
We will prove by contradiction.
Note that since $q_1(0)=\lambda+\kappa$ and $q_2(0)>\kappa+\varepsilon_3$, and the rates of change are bounded, in a sufficiently small neighborhood $[0,T_0]$ (depending only on $\varepsilon_3$), we have for all $t\in [0,T_0]$, (i) $q_1(t)\leq \lambda+\kappa+\varepsilon_3/2$, (ii)~$q_2(t)\geq \kappa+\varepsilon_3/2$, and
\[\mathrm{(iii)}~y_1(t)=q_1(t) - q_2(t)\leq \lambda-\frac{\varepsilon_3}{2}.\]
Since due to Claim~\ref{fact:mfl}, $q_1(t)$ is nondecreasing in $[0,T_0]$,
it is enough to produce a subinterval of $[0,T_0]$, where the right-derivative of $q_1(t)$ is bounded away from 0.
Now we will consider two cases:\\

\noindent
{\bf Case 1:} There exists $t'\in [0,T_0/2]$, such that $u(t')=0$ and $\delta_1(t')\leq\varepsilon_3/2$.
In this case, $\delta_0(t')>0$, and in a sufficiently small time interval almost all points (with respect to Lebesgue measure) are regular for $\delta_1(t)$.
Also, due to Proposition~\ref{prop:mf}, since for $t\geq t'$,
\[\delta_1(t)=\delta_1(t')+\xi(t)-\xi(t')-\nu\int_{t'}^t\delta_1(s)\dif s,\]
with $(d^+/dt)\xi(t)=\lambda-y_1(t)\geq \varepsilon_3/2$ at $t=t'$, we have for sufficiently small $t_1<T_0/4$ (where choice of $t_1$ does not depend on~$t'$), $\delta_1(t'+t_1)\geq t_1\varepsilon_3/4$.
Also, since the rate of decrease of $\delta_1(t)$ is bounded, there exists $t_2<T_0/4$ (where choice of $t_2$ does not depend on~$t'$ as well), such that, 
\[\delta_1(t)\geq \frac{t_1\varepsilon_3}{8}\quad\mbox{for all}\quad t\in [t'+t_1, t'+t_1+t_2]\subseteq [0,T_0].\] 
Thus, due to Claim~\ref{fact:mfl} there exists $\alpha>0$, such that during the time interval $[t'+t_1, t'+t_1+t_2]$, $(d^+/dt)q_1(t)\geq \min\{\alpha, \nu t_1\varepsilon_3/8\}$. 
Consequently,
\begin{equation}\label{eq:localq1case1}
q_1(T_0)\geq q_1(t'+t_1+t_2)\geq \lambda+\kappa+\min\Big\{\alpha, \frac{\nu t_1\varepsilon_3}{8}\Big\} t_2.
\end{equation}
It is important to note that the choices of $t_1$ and $t_2$ depend only on $\varepsilon_3$ and not on $t'$.

\noindent
{\bf Case 2:} For all $t\in [0,T_0/2]$, either $u(t)>0$ or $\delta_1(t)>\varepsilon_3/2$.
In this case, due to Claim~\ref{fact:mfl} (i) and (ii), there exists $\alpha>0$, such that $(d^+/dt)q_1(t)>\alpha$ for all $t\in [0,T_0/2]$.
Also, since $q_2(t)$ is non-decreasing in $[0,T_0]$.
we obtain
\begin{equation}\label{eq:localq1case2}
q_1(T_0) \geq q_1\Big(\frac{T_0}{2}\Big)\geq \lambda+\kappa+\frac{\alpha T_0}{2}.
\end{equation}

\noindent
Combining the two cases above, and choosing 
\[\varepsilon_4 = \min\Big\{\min\Big\{\alpha, \frac{\nu t_1\varepsilon_3}{8}\Big\} t_2,  \frac{\alpha T_0}{2}\Big\}>0\]
completes the proof of Claim~\ref{claim:q2conv}.
\end{proof}

\noindent
\textbf{Convergence of $\delta_1^N(\infty)$ and $\delta_0^N(\infty)$.}
Given the convergence of $q_1^N(\infty)$ and $q_2^N(\infty)$, the convergence of $\delta_1^N(\infty)$ and $\delta_0^N(\infty)$ can be seen immediately by observing that the mean-field limit has the following property:
\begin{claim}\label{claim:delta0delta1}
Starting from any state $\qq(0)\in E_\kappa$ with $q_1(0)=\lambda+\kappa$ and $q_2(0)=\kappa$, along any MFFSP $\delta_1(t)\to 0$ and $\delta_0(t)\to 1-\lambda-\kappa$ as $t\to\infty$.
\end{claim}
\noindent
The proof of Claim~\ref{claim:delta0delta1} is immediate from the description of the mean-field limit as in Proposition~\ref{prop:mf}, and hence is omitted.\\

\noindent
The proof of the statement in~\eqref{eq:foundbusy} follows by using the convergence of steady states and the PASTA property.
This completes the proof of Lemma~\ref{lem:steady-concentration}.

\section{Conclusion}\label{sec:con}
In this paper we studied the stability of systems under the TABS scheme and established large-scale asymptotics of the sequence of steady states.
Understanding stability of stochastic systems is of fundamental importance.
Systems under the TABS scheme, as it turned out, may be unstable for some $N$ even under a sub-critical load assumption.
As in many other cases, the lack of monotonicity makes the stability analysis much more challenging  from a methodological standpoint.
We developed a novel induction-based method and establish that the TABS scheme is stable for all large enough $N$.
The proof technique is of independent interest and potentially has a much broader applicability.
The key model-dependent part of our method is what can be called a weak monotonicity property, which ensures that for large enough $N$, with high probability, no matter where the system starts, in some fixed amount of time, there will be a certain fraction of busy servers.
Both traditional fluid limit (fixed $N$, initial state goes to infinity) and mean-field limit (for a sequence of processes with the number of queues $N \to \infty$) were used in an intricate manner to establish the results.

\section*{Acknowledgement}
DM was supported by The Netherlands Organization for Scientific Research (NWO) through Gravitation Networks grant 024.002.003, and TOP-GO grant 613.001.012.


\end{document}